\documentclass[11pt]{article}
\usepackage[margin = 1in]{geometry}

\usepackage{lipsum}
\usepackage{amsfonts}
\usepackage{graphicx}
\usepackage{epstopdf}
\usepackage[font=small]{caption}

\usepackage{mathtools}
\usepackage{mathrsfs}
\usepackage{empheq}
\usepackage{amsfonts}
\usepackage{amsgen,amsthm,amsmath,amstext,amsbsy,amsopn,amssymb,subfigure,stmaryrd}
\usepackage{comment}
\usepackage{enumitem}
\usepackage{natbib}
\usepackage{booktabs}
\usepackage{blkarray}
\usepackage{algorithm}
\usepackage{algpseudocode}
\usepackage{url}
\usepackage{longtable}
\usepackage{mathtools}
\usepackage{multirow}
\usepackage{xcolor}
\usepackage{dsfont}
\usepackage[
  colorlinks,
  citecolor=blue,
  linkcolor=red,
  anchorcolor=red,
  urlcolor=blue
]{hyperref}

\ifpdf
  \DeclareGraphicsExtensions{.eps,.pdf,.png,.jpg}
\else
  \DeclareGraphicsExtensions{.eps}
\fi

\usepackage{tikz}
\usetikzlibrary{decorations.pathreplacing,calligraphy}

\newtheorem{Definition}{Definition}
\newtheorem{Theorem}{Theorem}

\newtheorem{Lemma}{Lemma}
\newtheorem{Remark}{Remark}

\newtheorem{Proposition}{Proposition}

\newcommand{\bbR}{\mathbb{R}}

\newcommand{\cA}{\mathcal{A}}

\newcommand{\cD}{\mathcal{D}}
\newcommand{\cY}{\mathcal{Y}}
\newcommand{\cX}{\mathcal{X}}
\newcommand{\cE}{\mathcal{E}}

\newcommand{\hcY}{\widehat{\mathcal{Y}}}

\newcommand{\hT}{{\widehat{T}}}
\newcommand{\hr}{{\widehat{r}}}
\newcommand{\hf}{{\widehat{f}}}
\newcommand{\hy}{{\widehat{y}}}
\newcommand{\indi}{{\mathds{1}}}

\newcommand{\bbP}{\mathbb{P}}
\newcommand{\bbE}{\mathbb{E}}
\newcommand{\iidsim}{\stackrel{\textnormal{iid}}{\sim}}

\allowdisplaybreaks

\graphicspath{ {images/} }

\begin{document}
	\title{The Limits of Assumption-free Tests for Algorithm Performance}
\author{Yuetian Luo\thanks{Data Science Institute, University of Chicago} \ and Rina Foygel Barber\thanks{Department of Statistics, University of Chicago}}
	\date{}
	\maketitle

	\bigskip

\begin{abstract}
Algorithm evaluation and comparison are fundamental questions in machine learning and statistics---how well does an algorithm perform
at a given modeling task, and which algorithm performs best?
Many methods have been developed to assess algorithm performance, often based around cross-validation type strategies, retraining
the algorithm of interest on different subsets of the data and assessing its performance on the held-out data points. 
Despite the broad use of such procedures, the theoretical properties of these methods are not yet fully understood. In this work,
we explore some fundamental limits for answering these questions with limited amounts of data. In particular, we make a distinction between
two questions: how good is an algorithm $\cA$ at the problem of learning from a training set of size $n$, 
versus, how good is a particular fitted model produced by running $\cA$ on a particular 
training data set of size $n$? Our main results prove that, for any test that treats the algorithm $\cA$ as a ``black box'' (i.e., we can only study the behavior of $\cA$ empirically),
there is a fundamental limit on our ability to carry out inference on the performance of $\cA$, unless the number of available data points $N$
is many times larger than the evaluation sample size $n$ of interest. On the other hand, evaluating the performance of a particular fitted model
can be easy as long as the loss function is bounded and a holdout data set is available---that is, as long as $N-n$ is not too small.
We also ask whether an assumption of algorithmic stability might be sufficient to circumvent this hardness result. Surprisingly, we find that the same hardness result still holds for the problem of evaluating the performance of $\cA$, 
aside from a high-stability regime where fitted models are essentially nonrandom.  Finally, we also establish similar hardness results
for the problem of comparing multiple algorithms.
\end{abstract}

\section{Introduction}\label{sec:intro}
Evaluating the performance of a regression algorithm, and comparing the performance of different algorithms, are fundamental questions in machine learning and statistics \citep{salzberg1997comparing,dietterich1998approximate,bouckaert2003choosing,demvsar2006statistical,garcia2008extension,hastie2009elements,raschka2018model}. 
When performing a data analysis, if the properties of the underlying data distribution are unknown, how can we determine which algorithms would model the data well?
For instance, for the task of predicting a real-valued response $Y$, which algorithm would have the smallest possible error?

\subsubsection{Algorithm evaluation.} In general, given access to a limited amount of data drawn from an unknown distribution $P$ on $\cX\times\cY$, 
and given a particular algorithm $\cA$ (chosen by the analyst), we would like to ask:
\begin{equation*}
	\texttt{EvaluateAlg}\textnormal{: how well does algorithm } \cA \textnormal{ perform on data drawn from } P ?
\end{equation*}

Writing $\hf$ to denote the fitted model produced by running algorithm $\cA$ on a training data set, 
a related question is to ask about the performance of the trained model $\hf$:
\begin{equation*}
	\texttt{EvaluateModel}\textnormal{: how well does the fitted model } \hf \textnormal{ perform on data drawn from } P ?
\end{equation*}

Are these simply two different phrasings of the same question---or should we make a distinction between the two?

In practice, these two questions are often treated as interchangeable: for instance, a paper might 
state that ``Logistic regression performs well on this problem'' (which suggests that the question \texttt{EvaluateAlg} is being
addressed), but then justify the statement by using a holdout set to evaluate the fitted model $\hf$ produced by running logistic regression on a particular data set
(which is actually addressing the question \texttt{EvaluateModel}). 
In this work, we argue that \emph{these two questions are not equivalent}, and indeed, 
answering \texttt{EvaluateAlg} may be impossible even in settings where \texttt{EvaluateModel} is trivial to answer.

\subsubsection{Algorithm comparison.} In practice, rather than asking about the risk of a single algorithm, we might instead
be interested in comparing multiple algorithms to see which is best---for instance,
which algorithm (from a set of candidates) minimizes a certain measure of risk. To address this
setting, in this paper we will also consider questions of algorithm comparison. 
As for the problem of evaluation, we might pose a question about the algorithms themselves:
\begin{equation*}
	\texttt{CompareAlg}\textnormal{: which algorithm, } \cA_0 \textnormal{ or } \cA_1\textnormal{, performs better on data drawn from }P?
\end{equation*}

Or, we might ask about the trained models returned by these algorithms---if $\hf_0$ and $\hf_1$ are the fitted models produced
by algorithms $\cA_0$ and $\cA_1$, respectively, we ask:
\begin{equation*}
	\texttt{CompareModel}\textnormal{: which fitted model, } \hf_0 \textnormal{ or } \hf_1\textnormal{, performs better on data drawn from }P?
\end{equation*}

Again, while these two questions are sometimes treated as interchangeable, they are distinct in a meaningful way.
Indeed, which question is the ``right'' one to ask will necessarily depend on the goal of the analysis. For example,
a data analyst who seeks to build an accurate model, in order to then deploy this model for future predictions or some other task,
may be more interested in determining whether $\hf_0$ or $\hf_1$ will be more accurate for this task.
On the other hand, a researcher developing methodology in statistics or machine learning
who proposes a new algorithm will instead be more interested in comparing $\cA_0$ (their algorithm) against $\cA_1$ (an existing algorithm).
If the researcher creates a demo on a real data set to compare their method to existing work,
implicitly they are claiming that their method $\cA_0$ will perform better than the existing method $\cA_1$
on other data sets as well---they are not interested only in the particular fitted models produced in one specific experiment.
For a deeper discussion on the subtle distinctions between various questions we may ask when comparing (or evaluating)
algorithms, we refer the reader to \citet[Section 1]{dietterich1998approximate}.

\subsection{Problem formulation: evaluating an algorithm}
In this section, we will formalize our questions for algorithm evaluation (we will return to the problem of algorithm comparison later on).

To formalize the two questions \texttt{EvaluateAlg} and \texttt{EvaluateModel} defined above, and understand the distinction between the two, we now introduce some notation.
Let $\ell$ be a loss function, with $\ell(f(x),y)\geq 0$ measuring the loss incurred by regression
function $f$ on data point $(x,y)$. 
Formally, we can write $\ell: \hcY\times\cY\rightarrow\bbR_+$, where we are evaluating models $f:\cX\rightarrow\hcY$. 
For example:
\begin{itemize}
\item When predicting a real-valued response, we might use squared loss or absolute-value loss,
\[\ell(f(x),y) = (y - f(x))^2 \textnormal{ \ or \ }\ell(f(x),y) = |y-f(x)|,\]
with $\cY = \hcY = \bbR$.
\item If instead $\cY = \{0,1\}$ (a binary response), if $f(x) \in\hcY = \{0,1\}$ predicts the binary label then we might use the zero--one loss, 
\[\ell(f(x),y) = \indi_{f(x)\neq y}.\]
If instead we have $f(x)\in\hcY = [0,1]$, predicting the probability of a positive label (i.e., $f(x)$ estimates $\bbP(Y=1\mid X=x)$), we might use 
a hinge loss or logistic loss.
\end{itemize}
Given a particular choice of the loss function $\ell$, we define the risk of a model $f$ as 
\[R_P(f) = \bbE_P[\ell(f(X),Y)],\]
which is computed with respect to a test point sampled as $(X,Y)\sim P$. 
Now consider a regression algorithm $\cA$, which inputs a data set $\{(X_i,Y_i)\}_{i \in [n]}$ where $[n] := \{1, \ldots, n\}$ and returns a fitted model $\hf_n$. 
The fitted model $\hf_n$ has risk $R_P(\hf_n)$---this captures the expected loss for the future performance of $\hf_n$, on test
data drawn from $P$. 
Note we can write
\[R_P(\hf_n)=\bbE_P\left[\ell(\hf_n(X),Y)\mid \hf_n\right],\]
the expected loss of the trained model when we average over  the draw of a new 
test point $(X,Y)\sim P$, but  condition on the fitted model $\hf_n$ (or equivalently, on
the draw of the training data, $\{(X_i,Y_i)\}_{i\in[n]}$).
This risk is random: it depends on the fitted model $\hf_n$, which will vary depending on the random draw of the training data.
In contrast, we could also measure the performance of $\cA$ by asking about its expected risk, over a random draw of the training data as well as the test point:
\[R_{P,n}(\cA) = \bbE_P\left[\ell(\hf_n(X),Y)\right],\]
 where now the expected value is taken with respect to a random draw $(X_1,Y_1),\dots,(X_n,Y_n),(X,Y)\iidsim  P$---that is, with respect to a random draw
of both the training and test data. 
In particular, the model risk $R_P(\hf_n)$ and algorithm risk $R_{P,n}(\cA)$ are related as
\[R_{P,n}(\cA) = \bbE_P\left[R_P(\hf_n)\right].\]

With this notation in place, we can now make our questions, \texttt{EvaluateAlg} and \\\texttt{EvaluateModel}, more precise.
\begin{equation*}
	\texttt{EvaluateAlg}\textnormal{: what is } R_{P,n}(\cA)? \quad \textnormal{ versus }\quad  \texttt{EvaluateModel}\textnormal{: what is } R_P(\hf_n)?
\end{equation*}

These two questions are different: a low value of $R_{P,n}(\cA)$
tells us that $\cA$ tends to perform well for distribution $P$, while a low value of $R_P(\hf_n)$ tells us that \emph{this particular fitted model} 
produced by $\cA$ performs well. In particular, the latter statement could result from the fact that either $\cA$ is good or the particular data set at hand favors $\cA$. Which question is the ``right'' question to ask, will depend on the context,
and we do not claim that one of these questions is more important than the other; instead, the aim of our paper is to understand
the distinction between these two questions, and to emphasize that the two questions should not be treated interchangeably.

\subsubsection{Defining a hypothesis test.}
There are many possible goals we might define for addressing \texttt{EvaluateAlg} or \texttt{EvaluateModel}:
for example, we might wish to construct a confidence interval for the risk $R_{P,n}(\cA)$ or $R_P(\hf_n)$,
or to perform a hypothesis test assessing whether $R_{P,n}(\cA)$ or $R_P(\hf_n)$ lies above or below some prespecified threshold $\tau$,

\begin{equation}\label{eqn:hypothesis_test}\texttt{EvaluateAlg}: H_0:\, R_{P,n}(\cA) \geq \tau \textnormal{ \ versus \ } H_1: \,R_{P,n}(\cA) <\tau,\end{equation}
\begin{equation}\label{eqn:hypothesis_test2} 
	 \texttt{EvaluateModel}: H_0:\, R_P(\hf_n) \geq \tau \textnormal{ \ versus \ } H_1: \,R_P(\hf_n) <\tau.
\end{equation}
 We note that it is not hard to see that \texttt{EvaluateModel} is an easy problem as long as the loss function is bounded and we have a large holdout data. This is because after training model $\hf_n$, we can use the holdout data $\cD_{\textnormal{hold}}$ to estimate $R_P(\hf_n)$ via computing $ \frac{1}{|\cD_{\textnormal{hold}}|}\sum_{(X_i, Y_i) \in \cD_{\textnormal{hold}}} \ell(\hf_n(X_i), Y_i)$; if $|\cD_{\textnormal{hold}}|$ is large then this estimate is likely to be accurate. However, even in this bounded setting where \texttt{EvaluateModel} is relatively easy, we will see below that \texttt{EvaluateAlg} is fundamentally hard: our results will establish limitations on the ability of any test to answer the question posed by \texttt{EvaluateAlg}, even in the setting of a bounded loss. To establish such a hardness result, it suffices to consider the problem of performing
the hypothesis test~\eqref{eqn:hypothesis_test}. This is because, if we have instead constructed a confidence interval $\widehat{C}$
with a guarantee of containing the risk $R_{P,n}(\cA)$ with probability $\geq 1-\alpha$, we can test~\eqref{eqn:hypothesis_test}
at Type I error level $\alpha$ by simply checking whether $\widehat{C}\subseteq [0,\tau)$.

\subsection{Overview of contributions}
In this paper, for the problem of algorithm evaluation,
 we will establish fundamental limits on the ability of any \emph{black-box test} (that is,
any test that studies the behavior of $\cA$ empirically rather than theoretically)
to perform inference on the hypothesis test~\eqref{eqn:hypothesis_test} even when model evaluation is easy to solve. We will also see that this hardness result can be partially alleviated
by adding an assumption of algorithmic stability, but only in the high-stability regime.
We will also establish analogous results for the problem of algorithm comparison.

\textbf{Organization.}
In Section~\ref{sec:background}, we will present some background 
on algorithm testing and on algorithmic stability, in addition to giving an overview of some related work in the literature.
Section~\ref{sec:results-evaluation} presents our main results establishing the hardness
of testing \texttt{EvaluateAlg}, and establishes that an extremely simple and naive test 
already achieves (nearly) the maximum possible power. Results under an algorithmic stability
assumption are presented in Section~\ref{sec:results-evaluation-stability}.
We then turn to the problem of algorithm comparison, which is addressed in Section~\ref{sec:results-comparison}.
Discussion and conclusions are provided in Section~\ref{sec:discussion}. All proofs are deferred to the Supplementary Material \citep{luo2025limit}.

\section{Background} \label{sec:background}

In this section, we formalize some definitions and give additional background for studying the properties of algorithms.
Formally, for data lying in some space $\cX\times\cY$, we define an algorithm $\cA$ as a map
\[\cA: \  \cup_{n\geq 0} (\cX\times\cY)^n \rightarrow \big\{\textnormal{measurable functions $\cX\rightarrow\hcY$}\big\}.\]
That is, $\cA$ takes as input $\{(X_i,Y_i)\}_{i \in [n]}$, a data set of any size, and returns a fitted model $\hf_n: \cX\rightarrow\hcY$,
which maps a feature value $x$ to a fitted value, or prediction, $\hy = \hf_n(x)$.
For instance, for a real-valued response $Y\in\cY = \bbR$, we might take $\cA$ to be the least squares regression algorithm,
leading to linear fitted models $\hf_n$. 
In many settings, commonly used algorithms $\cA$ will include some form of randomization---for instance, stochastic gradient
descent type methods. To accommodate this setting, we will expand the definition of an algorithm to take an additional argument,
\begin{equation}\label{eqn:define_alg}\cA:  \ \cup_{n\geq 0} (\cX\times\cY)^n \times [0,1] \rightarrow \big\{\textnormal{measurable functions $\cX\rightarrow\hcY$}\big\}.\end{equation}
The fitted model $\hf_n = \cA\big(\{(X_i,Y_i)\}_{i \in [n]} ; \xi\big)$ is now obtained by running algorithm $\cA$
on data set $\{(X_i,Y_i)\}_{i \in [n]}$, with randomization provided by the argument $\xi\in[0,1]$, 
which acts as a random seed.\footnote{Throughout this paper, we implicitly assume measurability for all
functions we define, as appropriate. In the case of the algorithm $\cA$, this takes the following form: 
we assume the measurability of the map
$\left( (x_1, y_1), \ldots, (x_n, y_n), \xi, x \right) \mapsto [\cA(\{(x_i,y_i)\}_{i\in[n]};\xi)](x)$.} Of course, the general definition~\eqref{eqn:define_alg}
of a randomized algorithm also accommodates the deterministic (i.e., non-random) setting---we can simply choose an algorithm $\cA$
whose output depends only on the data set $\{(X_i,Y_i)\}_{i\in [n]}$, and ignores the second argument $\xi$. 

With this notation in place, we are now ready to consider the question \texttt{EvaluateAlg}. To formalize the definition of risk $R_{P,n}(\cA)$ introduced in Section~\ref{sec:intro}
 in a way that  accommodates the case of randomized algorithms as in~\eqref{eqn:define_alg}, we will define the risk $R_{P,n}(\cA)$ as follows:
\begin{equation}\label{eqn:define_risk}
R_{P,n}(\cA) = \bbE\left[\ell\left( \hf_n(X_{n+1}) , Y_{n+1}\right)\right]\textnormal{ for }\hf_n = \cA\big(\{(X_i,Y_i)\}_{i \in [n]} ; \xi\big),\end{equation}
where the expected value is taken with respect to the distribution of data points $(X_i,Y_i)\iidsim  P$ for $i = 1,\ldots, n+1$
and an independent random seed $\xi\sim\textnormal{Unif}[0,1]$.

\subsection{Black-box tests}

To understand the difficulty of the question \texttt{EvaluateAlg}, we need to 
formalize the limitations on what information is available to us, as the analyst, for performing inference on an algorithm's risk.
For example, if we know the distribution $P$ of the data, and know that our algorithm $\cA$ is given by $k$-nearest-neighbors (kNN) regression,
then we can derive a closed-form expression for the risk $R_{P,n}(\cA)$ theoretically.
Clearly, this is not the setting of interest, since in practice we cannot rely on assumptions about the data
distribution---but in fact, even if $P$ is unknown, estimating $R_{P,n}(\cA)$ for a simple algorithm such as kNN
is relatively straightforward. In modern settings, however, a state-of-the-art algorithm $\cA$ is typically far more complex,
and may be too complicated to analyze theoretically. For this reason, in this work we consider \emph{black-box} procedures \citep{myers2011art},
which analyze the algorithm $\cA$ via empirical evaluations, rather than theoretical calculations.
The intuition is this: if $\cA$ is implemented as a function in some software package, our procedure
is allowed
to call the function (perhaps as many times as we like, or perhaps with a computational budget), but is not able 
to examine the implementation of $\cA$ (i.e., we cannot read the code that defines the function).

We formalize this black-box setting in the following definition, which is based on a similar construction given in \cite{kim2021black}.
\begin{Definition}[Black-box test for algorithm evaluation] \label{def:black-box-test} 
Let $\mathfrak{D} = \cup_{m \geq 0} (\cX \times \cY)^m$ denotes the space of data sets of any size. Consider any function $\hT$ that takes as input an algorithm $\cA$ and a dataset $\cD \in \mathfrak{D}$,
and returns a (possibly randomized) output $\hT(\cA,\cD)\in\{0,1\}$. 
We say that $\hT$ is a black-box test if it can be defined as follows:\footnote{For this definition, we assume measurability of
the map $(\cD,\zeta^{(1)},\zeta^{(2)},\dots,\zeta)\mapsto \hT(\cA,\cD)$, which inputs the initial data set $\cD$ and the
random seeds used throughout the construction of the test, and returns the output of the test $\hT(\cA,\cD)\in\{0,1\}$.}  
 for some functions $g^{(1)},g^{(2)},\dots,g$, and for random seeds, $\zeta^{(1)},\zeta^{(2)},\dots,\zeta\iidsim  \textnormal{Unif}[0,1]$,
 		 \begin{enumerate}
	 	\item At stage $r =1$, generate a new dataset and a new random seed,
		\[(\cD^{(1)},\xi^{(1)} ) = g^{(1)} (\cD, \zeta^{(1)} ),\]
		and compute the fitted model $\hf^{(1)} = \cA( \cD^{(1)}; \xi^{(1)} )$.
	 	\item For each stage $r = 2,3,\ldots$, generate a new dataset and a new random seed,
		\[(\cD^{(r)},\xi^{(r)} ) = g^{(r)} \big(\cD, (\cD^{(s)})_{1\leq s<r}, (\hf^{(s)})_{1\leq s<r}, (\zeta^{(s)})_{1\leq s\leq r} , (\xi^{(s)})_{1\leq s<r}\big),\]
		and compute the fitted model
		$\hf^{(r)} = \cA(\cD^{(r)};\xi^{(r)})$.
		\item Repeat until some stopping criterion is reached; let $\hr$ denote the (data-dependent) total number of rounds.
	 	\item Finally, return
		\[\hT(\cA,\cD) = g\big(\cD,(\cD^{(r)})_{1\leq r\leq \hr},(\hf^{(r)})_{1\leq r\leq \hr},(\zeta^{(r)})_{1\leq r\leq \hr},(\xi^{(r)})_{1\leq r\leq \hr},\zeta\big).\]
		\end{enumerate}
\end{Definition}

In this definition, at each stage $r$ we are allowed to deploy the training algorithm $\cA$ on a data set $\cD^{(r)}$
that we have designed adaptively---e.g., by subsampling from the original data set $\cD$, or by generating simulated data given the past information.
The random seeds $\zeta^{(1)}, \zeta^{(2)}, \ldots, \zeta $ allow for randomization at each step of the procedure, if desired---for instance,
to allow for randomly subsampling from the original data set $\cD$. (Of course,
we might choose to run a deterministic procedure, as a special case.)

\subsubsection{An example of a black-box test.}
As an example of the type of test that would satisfy Definition~\ref{def:black-box-test}, we can consider a cross-validation (CV) \citep{stone1974cross,geisser1975predictive} based estimate of risk. Assuming we have $N$ data points $\{(X_i,Y_i)\}_{i\in [N]}$, $K$ divides $N$ and $n = N(1-1/K)$ for some integer $K\geq 1$,
we can construct an estimator as follows:
\begin{itemize}
\item Construct a partition $[N] = S_1\cup\dots\cup S_K$ with $|S_k| = N/K$ for each $k$.
\item For each $k=1,\dots,K$, define
\[\hf^{(k)} = \cA(\cD_{\backslash S_k}; \xi_k)\]
where $\cD_{\backslash S_k} = \{(X_i,Y_i)\}_{i\in [N]\backslash S_k}$ is a data set of size $n$,
 and where $\xi_k\sim\textnormal{Unif}[0,1]$. Then estimate
the risk of each fitted model,
\[\widehat{R}(\hf^{(k)}) = \frac{1}{|S_k|}\sum_{i\in S_k}\ell(\hf^{(k)}(X_i),Y_i).\]
\item Finally, average over all fitted models to define the CV estimate of risk:
\begin{equation}\label{eqn:Rhat_CV}\widehat{R}^{\textnormal{CV}}_n(\cA) = \frac{1}{K}\sum_{k=1}^K \widehat{R}(\hf^{(k)}),\end{equation}
and return a 1 or a 0 depending on whether $\widehat{R}^{\textnormal{CV}}_n(\cA)< \tau$ or $\widehat{R}^{\textnormal{CV}}_n(\cA)\geq\tau$.
\end{itemize}

\subsubsection{An example that is not black-box.} Tests that require
  infinitely many executions of a black-box algorithm are not admissible under Definition~\ref{def:black-box-test}. For example,  if $\cX = \bbR^d$, then given a data set $\cD_N$ and an additional
test point $(x,y)$, we would not be able
to answer the question, ``what is the maximum loss that could be incurred by adding a single data point to the data set''---that is,
a black-box test would be unable to calculate
\[\sup_{(x',y')\in\cX\times\cY} \ell(\hf_{(x',y')}(x),y) \textnormal{ where }\hf_{(x',y')} = \cA\big(\cD_N\cup\{(x',y')\};\xi\big).\]
This is because 
evaluating the supremum requires training the model $\hf_{(x',y')}$ for each $(x',y')\in\cX\times\cY$, which cannot be accomplished
with any finite collection of calls to $\cA$.
(On the other hand, computing the maximum loss over any test point,
\[\sup_{(x,y)\in\cX\times\cY} \ell(\hf(x),y) \textnormal{ where }\hf = \cA\big(\cD_N\big),\]
is possible, since this can be done with only a single call to $\cA$; even though this supremum implicitly requires
 infinitely many evaluations of $\hf$, this is permitted since the functions $g^{(1)},g^{(2)},\dots,g$ in 
Definition~\ref{def:black-box-test} are allowed to depend arbitrarily on the fitted models $\hf^{(1)},\hf^{(2)},\dots$. We make this distinction because, in practice, it is often the case that $\cA$ is complex but a
 typical  fitted model $\hf$
 is simple---e.g., a sophisticated variable selection procedure that results in a linear regression model.)

\subsection{Additional related literature}
Many statistical tests and heuristic methods have been proposed for algorithm evaluation and comparison. Given a large dataset, 
 using a holdout set, also known as sample splitting, is often recommended to evaluate or compare algorithms in practice \citep{hinton1995assessing,neal1998assessing} (note that using a holdout set to evaluate a fitted model $\hf_n$ is directly
 addressing the question \texttt{EvaluateModel}, rather than \texttt{EvaluateAlg}).
 With a limited sample size, resampling methods are more popular, including methods
such as cross-validation, and its variants such as $5 \times 2$-fold cross-validation (CV) \citep{dietterich1998approximate}, (corrected) repeated train-validation \citep{nadeau1999inference}, and nested cross-validation \citep{varma2006bias,raschka2018model}. 
As discussed earlier, many works in the literature have emphasized the qualitative difference between \texttt{EvaluateAlg} and \texttt{EvaluateModel} (e.g., \cite{dietterich1998approximate,hastie2009elements}). Which of these questions is more useful to answer, may depend on the particular
application; for example, \cite{trippe2023confidently} argued that \texttt{EvaluateModel} may be a more meaningful question in certain settings. 
However, in settings where \texttt{EvaluateAlg} is the question of interest, how to perform inference on this target is less well understood.
\cite{nadeau1999inference} identified that ignoring the variability due to the randomness of the training set can severely undermine the validity of inference in \texttt{EvaluateAlg}. Classical results on CV have suggested that CV can consistently determine which of two algorithms
 is better, but these results are generally asymptotic and require some restrictive conditions on the data distribution \citep{shao1993linear,yang2007consistency}.  See also \cite{hastie2009elements,arlot2010survey} for surveys on model selection via CV. Based on the asymptotic normality of CV error, the recent work \cite{bayle2020cross} discussed how to compare the prediction error of averaged models of two algorithms (in the same spirit as \texttt{EvaluateModel}) in the CV procedure. Similar results have been also extended to compare the algorithm's risk \citep{austern2020asymptotics,li2023asymptotics}. Understanding the estimand of CV and effectiveness of CV in evaluating a fitted model are also active areas of research \citep{bates2021cross,iyengar2025cross}. 
There are also many other methods available for estimating the prediction error of an algorithm other than CV; see, e.g., \cite{rosset2019fixed} and references therein
for work in a nonparametric regression setting.

Our work also contributes to an important line of work on impossibility or hardness results in statistics, especially in the distribution-free inference setting \citep{bahadur1956nonexistence,barber2021limits,shah2020hardness,barber2020distribution,medarametla2021distribution,kim2022local}. The closest work to ours is \cite{kim2021black}, 
which establishes a hardness result for the problem of constructing an assumption-free test of algorithmic stability. 

Finally, the literature on algorithmic stability is also very relevant to this work---we will give background
on this area of the literature in Section~\ref{sec:background_stability}, below.

\section{Theoretical results for algorithm evaluation} \label{sec:results-evaluation}
In this section, we present our first main result, establishing the hardness of answering the question \texttt{EvaluateAlg}---or more concretely,
of performing the hypothesis test~\eqref{eqn:hypothesis_test}---via a black-box procedure.
With the aim of avoiding placing any assumptions on the algorithm $\cA$ and on the distribution $P$,
we would like to provide an answer to \texttt{EvaluateAlg} that is uniformly valid over \emph{all} possible $\cA$ and $P$.
Specifically, fixing an error tolerance level $\alpha$, an evaluation sample size $n$, an available dataset $\cD_N = \{(X_i,Y_i)\}_{i\in [N]}\iidsim P$ of size $N$, and a risk 
threshold $\tau > 0$ for our hypothesis test, we will consider tests $\hT = \hT(\cA,\cD_N)$ 
that satisfy 
\begin{equation}\label{eqn:validity}\bbP_P( \hT(\cA,\cD_N) = 1)\leq \alpha \textnormal{ for any $\cA,P$ such that $R_{P,n}(\cA)\geq \tau$},\end{equation}
where the probability is taken with respect to the data set $\cD_N$ (as well as any randomization in the test $\hT$).
From this point on, we treat the allowed error level $\alpha\in(0,1)$, the risk threshold $\tau > 0$, and sample sizes $n,N\geq 1$ as fixed; 
our results will give finite-sample (non-asymptotic) limits on power for this testing problem. We will see that power for any valid black-box
test for solving \eqref{eqn:hypothesis_test} is limited unless the available sample size $N$ is much larger than $n$.

\subsection{Limits of black-box tests for bounded loss} \label{sec:bounded-loss}
In this section, we
establish limits on power for any universally valid black-box test on the testing problem \eqref{eqn:hypothesis_test}, in a bounded loss setting.

We first define some additional notation. For a distribution $P$ on $\cX\times\cY$, let
\[R_P^{\max} = \sup_f R_P(f) ,\]
the largest possible risk of any function $f$ for data drawn from $P$. 
We are primarily interested in testing the hypothesis~\eqref{eqn:hypothesis_test} for a threshold $\tau\leq R_P^{\max}$
(because if we were instead considering $\tau>R_P^{\max}$, this would be a fundamentally simpler setting---\emph{any} algorithm
will have risk $<\tau$ over distribution $P$, so the hypothesis test would depend only on $P$,  not on the algorithm $\cA$).

We are now ready to state the theorem.

\begin{Theorem} \label{thm:limits_evaluate} 
Assume that either $|\cX| = \infty$ or $|\cY| = \infty$,
and that the loss $\ell$ takes values in $[0,B]$. 
Let $\hT$ be a black-box test (as in Definition \ref{def:black-box-test}),
and assume that $\hT$ satisfies the assumption-free validity condition~\eqref{eqn:validity}.
Let $\tilde\tau = \tau(1 + \frac{\alpha^{-1}-1}{N} ) $. 

Then the power of $\hT$ to detect low risk is bounded as follows:
for any $\cA$ and any $P$ with $R_{P,n}(\cA)< \tau$,
\[
	\bbP_P( \hT(\cA, \cD_N)  = 1) \leq \left[\alpha\left(1 + \frac{ \tilde\tau - R_{P,n}(\cA) }{R_{P}^{\max} - \tilde\tau } \right)^{N/n} \right] \wedge 1,
\]
if we assume $\tilde\tau < R_P^{\max}$ so that the denominator is positive.
\end{Theorem}

 In particular, as expected, if the risk $R_{P,n}(\cA)$ is only slightly lower than $\tau$, then the power is 
only slightly higher than $\alpha$ (i.e., no better than random). But, even if the risk is low (e.g., $R_{P,n}(\cA)=0$ at the extreme),
if the available data is limited (e.g., $N$ is close to the evaluation sample size $n$), 
then we see that the power can only be a constant multiple of $\alpha$. If the allowed error level $\alpha$ is small,
we therefore cannot have a powerful test for \emph{any} alternative. On the other hand, under the same setting as in Theorem \ref{thm:limits_evaluate}, we should clarify that \texttt{EvaluateModel} is relatively easy, as a test leveraging the holdout data can solve the problem. For instance if $N = n+m$, then we have a holdout set of $m$ many data points with
 which we can estimate $R_P(\hf_n)$ up to error $\mathcal{O}_P(\frac{1}{\sqrt{m}})$ by concentration---and therefore, testing whether $R_P(\hf_n) < \tau$ is straightforward in the black-box setting. We would like to emphasize that Theorem \ref{thm:limits_evaluate} is a hardness result and it is not obvious a priori given many successful examples in distribution-free inference \citep{vovk2005algorithmic,angelopoulos2024theoretical}. This suggests that it may be interesting to consider different definitions of assumption-free validity, since the definition of validity given in \eqref{eqn:validity} leads to the above hardness results and consequently appears to be too strong for meaningful inference to be possible.

In addition, the above result has a few other important implications. First, it complements the results in \cite{dietterich1998approximate,nadeau1999inference} and provides more quantitative results on the difference between \texttt{EvaluateAlg} and \texttt{EvaluateModel}. Second, it reveals that in the assumption-free setting with a limited number of samples, any claims about the risk of a procedure should be interpreted as a claim about the risk of the fitted model rather than the risk of the algorithm. Third, it reveals a trade-off in our ability to perform inference on \texttt{EvaluateAlg}. This is because given a data set size $N$, if we wish to have high power for identifying whether an algorithm $\cA$ has low risk $R_{P,n}(\cA)$,
we need to choose a small sample size, $n\ll N$---but
for small $n$, the algorithm $\cA$ will likely return fitted models that are less accurate, i.e.,
the risk $R_{P,n}(\cA)$ itself will likely be higher. (This type of tradeoff has been described by 
\cite{rinaldo2019bootstrapping}  as the \emph{inference--prediction tradeoff}, in a related setting.)

\begin{Remark} \label{rem:valueB}
	We note that the bound for the loss function $B$ does not explicitly appear in Theorem \ref{thm:limits_evaluate}, but the value of $\tau, R_{P,n}(\cA)$ and $R_{P}^{\max}$ typically scales with $B$. For example, in common settings, if we multiply the loss function by $10$, then $\tau, R_{P,n}(\cA)$ and $R_{P}^{\max}$ will also increase by a multiplicative factor $10$. But in that case, the ratio $\frac{ \tilde\tau - R_{P,n}(\cA) }{R_{P}^{\max} - \tilde\tau }$ in Theorem \ref{thm:limits_evaluate} would not change. In that sense, the power upper bound we get is scale-invariant. 
\end{Remark}

\begin{Remark}[Validity for randomized versus deterministic algorithms]
The results of Theorem~\ref{thm:limits_evaluate} bound the power of any black-box test $\hT$
when applied to any algorithm $\cA$, including algorithms that may be randomized (i.e., $\cA(\cD;\xi)$ has nontrivial
dependence on the random seed $\xi$). This is a stronger conclusion than if we only bounded power
for deterministic algorithms, but the assumption is stronger as well: the validity condition~\eqref{eqn:validity}
requires $\hT$ to have a bounded error rate over all $\cA$ (including algorithms that are randomized).
We remark that the deterministic version of the result also holds: that is,
if we require validity~\eqref{eqn:validity} only with respect to deterministic $\cA$ (algorithms for which $\cA(\cD;\xi)$
does not have any dependence on $\xi$), then the bound on power given in Theorem~\ref{thm:limits_evaluate}
also holds for all deterministic $\cA$ (in fact, the same construction in the proof can be used for both versions of the result).
\end{Remark}

\begin{Remark}[The infinite cardinality assumption] \label{rem:cardinality-assumption}
In this theorem (and in all our results below), we assume that either $|\cX|=\infty$
or $|\cY|=\infty$, i.e., the data points lie in a space $\cX\times\cY$ that has infinite cardinality.
The reason for this assumption is that, for finite $\cX\times\cY$, since the definition of a black-box test does not place any computational
constraints on the number of calls to $\cA$, we can perform an exhaustive search
to fully characterize the behavior of $\cA$. 
Therefore, the problem of testing the risk of $\cA$ \emph{with no computational constraints} is meaningful only for an infinite-cardinality
space. On the other hand, in practice computational limits are a natural component of any testing problem, and therefore
inference on risk for a large but finite space $\cX$ likely will have power limited by some function of the computational budget.
We leave this more complex question for future work. (See also Appendix~B.1 for further discussion.)
\end{Remark}

Careful readers may wonder does the result in Theorem \ref{thm:limits_evaluate} hold for testing other notions of algorithm risk. One possible choice is the following PAC-type notion: given any $\delta \in (0,1), \epsilon > 0$, consider
\begin{equation} \label{eq:tail-risk-test}
	H_0: \bbP( R_P(\cA(\cD_n; \xi)) \geq \epsilon ) \geq \delta \quad \textnormal{versus} \quad H_1:  \bbP( R_P(\cA(\cD_n; \xi)) \geq \epsilon ) < \delta. 
\end{equation} Then we can ask what is the best power we can hope for if the test is required to satisfy the following notion of assumption-free validity:
\begin{equation} \label{ineq:new-validity}
	\bbP_P( \hT(\cA,\cD_N) = 1)\leq \alpha \textnormal{ for any $\cA,P$ such that $\bbP( R_P(\cA(\cD_n; \xi)) \geq \epsilon ) \geq \delta$}.
\end{equation}
It turns out that the hardness result established in Theorem \ref{thm:limits_evaluate} also transfers to the new testing problem \eqref{eq:tail-risk-test}. 
\begin{Proposition}\label{thm:limits_evaluate-new-validity} 
Assume that either $|\cX| = \infty$ or $|\cY| = \infty$,
and that the loss $\ell$ takes values in $[0,B]$. 
Let $\hT$ be a black-box test (as in Definition \ref{def:black-box-test}),
and assume that $\hT$ satisfies the assumption-free validity condition~\eqref{ineq:new-validity}.
Let $\tau = \delta B + \epsilon$ and $\tilde\tau = \tau(1 + \frac{\alpha^{-1}-1}{N} ) $. 

Then the power of $\hT$ to detect low risk in the sense of \eqref{eq:tail-risk-test} is bounded as follows:
for any $\cA$ and any $P$ with $\bbP( R_P(\cA(\cD_n; \xi)) \geq \epsilon ) < \delta$,
\[
	\bbP_P( \hT(\cA, \cD_N)  = 1) \leq \left[\alpha\left(1 + \frac{ \tilde\tau - R_{P,n}(\cA) }{R_{P}^{\max} - \tilde\tau } \right)^{N/n} \right] \wedge 1,
\]
if we assume $\tilde\tau < R_P^{\max}$ so that the denominator is positive.
\end{Proposition}

\subsection{A matching bound}\label{sec:matching_binomial}
We will now show that the upper bound on the power given in Theorem~\ref{thm:limits_evaluate} can be approximately
achieved, in certain settings, with a simple test that is based on splitting the data into independent batches.
For this result, we will consider the case that the loss $\ell$ takes values in $\{0,1\}$.
For example, in the case of a categorical response $Y$ (such as a binary label),
we might use the zero--one loss,
\[\ell(\hy,y) = \indi_{\hy\neq y},\]
while if the response $Y$ is real-valued, we might instead 
use a thresholded loss,
\[\ell(\hy,y) = \indi_{|\hy - y|>r},\]
which indicates whether the error of the prediction exceeds some allowed threshold $r$.

For this setting, we now construct a simple test based on splitting the data  $\cD_N$ into independent batches of size $n+1$, as follows.
For each batch $j = 1,\dots, \lfloor \frac{N}{n+1}\rfloor$, define a fitted model
\[\hf^{(j)} = \cA(\cD^{(j)};\xi^{(j)}) \textnormal{ where } \cD^{(j)} = \{(X_i,Y_i)\}_{i = (j-1)(n+1) + 1}^{j(n+1) - 1}\textnormal{ and }\xi^{(j)}\iidsim\textnormal{Unif}[0,1],\]
and evaluate the loss on the last data point of this batch,
\[\ell^{(j)} = \ell\big(\hf^{(j)}(X_{j(n+1)}),Y_{j(n+1)}\big) \in\{0,1\}.\]
Note that, by construction, $\ell^{(j)}$'s are i.i.d.\ draws from a Bernoulli distribution with parameter $R_{P,n}(\cA)$.
Therefore, we can define our test by comparing the sum $S = \sum_{j=1}^{\lfloor N/(n+1)\rfloor} \ell^{(j)}$ against a Binomial distribution: 
\begin{equation}\label{eqn:binomial-test}
	\begin{split}
		\hT_{\textnormal{Binom}}(\cA,\cD_N) = \begin{cases} 
1, & S < k_*,\\
1, & S = k_* \textnormal{ and }\zeta\leq a_*,\\
0, & \textnormal{ otherwise},
\end{cases}
	\end{split}
\end{equation}
for a random value $\zeta\sim\textnormal{Unif}[0,1]$,
where the nonnegative integer $k_*$ and value $a_*\in[0,1)$ are chosen as the unique values satisfying
\[
	\bbP( \textnormal{Binomial}( \textstyle{\lfloor \frac{N}{n+1}\rfloor}, \tau  )  < k^* ) + a^*\cdot  \bbP( \textnormal{Binomial}( \textstyle{\lfloor \frac{N}{n+1}\rfloor}, \tau  )  = k^* ) = \alpha.
\]

\begin{Theorem} \label{thm:binomial_power} 
The test $\hT_{\textnormal{Binom}}$ is a black-box test (as in Definition \ref{def:black-box-test}),
and satisfies the assumption-free validity condition~\eqref{eqn:validity}. Moreover,
if $\alpha < (1-\tau)^{\lfloor \frac{N}{n+1}\rfloor}$, then
for any $\cA$ and any $P$ with $R_{P,n}(\cA)<\tau$, the power of $\hT_{\textnormal{Binom}}$ is equal to
\[	\bbP( \hT_{\textnormal{Binom}}(\cA,\cD_N)  = 1) = \alpha\left(1 + \frac{ \tau - R_{P,n}(\cA) }{1 - \tau } \right)^{\lfloor N/(n+1)\rfloor} \wedge 1. \]
\end{Theorem}

Comparing this power calculation to the result of Theorem~\ref{thm:limits_evaluate},
we can see that in a setting where $n,N$ are large and $R_P^{\max} \approx1$, which is typically the case for loss functions mentioned above as for each $\widehat{y}$, there is always a $y$ such that $\ell(\hy,y) = 1$, the power of the simple test $\hT_{\textnormal{Binom}}$
is nearly as high as the optimal power that can be achieved by any black-box test. 
This may seem counterintuitive: how can this simple test be (nearly) optimal,
given that it uses the data very inefficiently---without any more sophisticated strategies such as cross-validation,
resampling, synthetic or perturbed data, etc? Nonetheless, comparing Theorem \ref{thm:limits_evaluate} and Theorem \ref{thm:binomial_power}
reveals that the simple Binomial test, based only on sample splitting,
is nearly optimal because of the fundamental hardness of the problem.

\subsubsection{A concrete example.} To take a simple example to make this concrete,
 consider a binary response setting where $Y = f_*(X) \in\{0,1\}$ (i.e.,
there is an oracle classifier $f_*$ with perfect accuracy), with $\bbP_P(Y=0)= \bbP_P(Y=1) = 0.5$, i.e., each label is equally likely.
Fix $\tau=0.5$, so that our hypothesis test~\eqref{eqn:hypothesis_test} is simply asking whether the performance of $\cA$ is better
than random, and suppose we take $N=n+1$.
Then by Theorem~\ref{thm:binomial_power}, we see that the power of the simple test $\hT_{\textnormal{Binom}}$
is equal to $(2 - 2R_{P,n}(\cA))\alpha$, which is $2\alpha$ at most---a very low power when $\alpha$ is small. However,
we can compare this to 
Theorem~\ref{thm:limits_evaluate}, which tells us that it is impossible for \emph{any} black-box test to have power higher than $\approx (2-2R_{P,n}(\cA))^{1+1/n}\alpha$,
again at most $\approx 2\alpha$.
On the other hand, if $N/n$ is large, the power achieved by the test $\hT_{\textnormal{Binom}}$ might be close to 1.

\subsection{A note for unbounded loss} \label{sec:unbounded-loss}
Curious readers may wonder what would happen if the loss is unbounded, i.e., $\sup_{\hy\in\hcY,y\in\cY}\ell(\hy,y) = \infty$. It is not difficult to convenience ourselves that when the loss is unbounded, the problem becomes even harder. For \texttt{EvaluateAlg}, we can show that under the same setting as in Theorem \ref{thm:limits_evaluate}, no black-box test can have better-than-random power for inference on the risk of \emph{any} algorithm $\cA$.

\begin{Theorem} \label{thm:limits_evaluate_unbounded} 
Assume that either $|\cX| = \infty$ or $|\cY| = \infty$,
and that the loss $\ell$ is unbounded, i.e., $\sup_{\hy,y}\ell(\hy,y) = \infty$.
Let $\hT$ be a black-box test (as in Definition \ref{def:black-box-test}),
and assume that $\hT$ satisfies the assumption-free validity condition~\eqref{eqn:validity}.
Then the power of $\hT$ to detect low risk is bounded as follows:
for any $\cA$ and any $P$ with $R_{P,n}(\cA)< \tau$,
\[
	\bbP_P( \hT(\cA, \cD_N)  = 1) \leq \alpha.
\]
\end{Theorem}

At the same time, when the loss is unbounded, assumption-free inference for \texttt{EvaluateModel} also becomes trickier, as the concentration for the estimator $ \frac{1}{|\cD_{\textnormal{hold}}|}\sum_{(X_i, Y_i) \in \cD_{\textnormal{hold}}} \ell(\hf_n(X_i), Y_i)$ could fail. {During the revision, one of the reviewers pointed out that under a slightly stronger assumption on the marginal distribution of $X$, i.e., assuming it is non-atomic, then assumption-free inference for the model risk is also impossible when the loss is unbounded.} In this scenario, we also need additional assumptions to infer model risk, see some examples in \cite{bousquet2002stability,celisse2016stability}.

On the other hand, we also find that the bounded loss assumption may not be essential in separating the difficulty of \texttt{EvaluateAlg} and \texttt{EvaluateModel}. In particular, if we slightly change the inference target from mean risk to ``median" risk, then the difference of \texttt{EvaluateAlg} and \texttt{EvaluateModel} will manifest again. We refer readers to Appendix A for more details.

\section{The role of algorithmic stability} \label{sec:results-evaluation-stability}

 Theorem~\ref{thm:limits_evaluate} suggests that to alleviate the hardness result for algorithm evaluation, we must consider placing assumptions on either the data distribution $P$ or the algorithm $\cA$, since 
any universally valid test $\hT$ must necessarily have low power. 

A large body of recent works suggests that in many contexts, algorithm stability helps CV-based methods achieve distribution-free guarantees \citep{kumar2013near,barber2021predictive,austern2020asymptotics,bayle2020cross}. 
This suggests that this might also be the case for the problem of addressing \texttt{EvaluateAlg}: is estimating the risk $R_{P,n}(\cA)$ achievable with a stability assumption, and no other conditions?
In this section, we will explore this possibility. Perhaps surprisingly, we will see that the hardness result of Theorem~\ref{thm:limits_evaluate} 
remains the same, except in an extremely high-stability regime. 

\subsection{Background on algorithmic stability}\label{sec:background_stability}
At a high level, \emph{algorithmic stability} expresses the idea that the fitted model produced
by an algorithm should not change massively under a slight perturbation to the input training data.
Algorithmic stability has long been a central notion in statistics and machine learning. Intuitively,
it is naturally a desirable property for learning algorithms---but it also leads to many downstream implications, as well.
For instance, it is known that stable algorithms tend to generalize better \citep{bousquet2002stability,shalev2010learnability}. It also plays a critical role in a number of different problems, such as model selection \citep{meinshausen2010stability}, reproducibility \citep{yu2013stability}, and predictive inference \citep{steinberger2018conditional,barber2021predictive}. Stability is known to hold for many simple learning algorithms such as $k$-nearest neighbors \citep{devroye1979distribution} and ridge regression \citep{bousquet2002stability}. It has also been shown to hold for any algorithm whose fitted models
are aggregated via bootstrapping or bagging \citep{breiman1996bagging,elisseeff2005stability,soloff2023bagging}.

There are also connections between stability and our ability to assess algorithmic performance;
for instance, \citet[Section 1]{dietterich1998approximate}, in their discussion of the distinction between questions of assessing
models versus assessing algorithms, pointed out that addressing questions such as \texttt{EvaluateAlg}
must ``rely on the assumption that the performance of learning algorithms changes smoothly with changes
in the size of the training data''. A recent line of literature gives precise quantitative results
for the problem of estimating risk under strong stability conditions: specifically,
recent work by \cite{austern2020asymptotics} establishes asymptotic normality
of the cross-validation estimator $\widehat{R}^{\textnormal{CV}}_n(\cA)$. 
Related work by \cite{bayle2020cross} also establishes asymptotic normality of $\widehat{R}^{\textnormal{CV}}_n(\cA)$ under 
weaker stability conditions.

On the other hand, \citet{kim2021black} established a hardness result for the problem of proving empirically that a black-box algorithm is stable (i.e., establishing, via a black-box test, that stability holds for a given algorithm $\cA$ whose theoretical properties are unknown).

\subsubsection{Defining algorithmic stability.}
There are many different definitions of algorithmic stability in the literature (see, e.g., \cite{bousquet2002stability} and Appendix A of \cite{shalev2010learnability} for some comparisons). In this work, we will consider the following notion of stability: 
  \begin{Definition}[Algorithmic stability] \label{def:ell_q_stability}
 Let $\cA$ be any algorithm, and let $P$ be any distribution. Fix sample size $n\geq 1$.
Then we define the $\ell_q$-stability of the triple $(\cA,P,n)$ as
 \[
 		\beta_q(\cA,P,n):=\max_{j=1,\ldots,n} \bbE\left[ \left| \ell(\hf_n(X_{n+1}), Y_{n+1}) - \ell(\hf_n^{- j}(X_{n+1}), Y_{n+1})\right|^q \right] ^{1/q},
\]
 where the trained models are given by
 \[\hf_n = \cA\left(\{(X_i,Y_i)\}_{i\in[n]}; \xi\right), \quad \hf_n^{-j} = \cA\left(\{(X_i,Y_i)\}_{i\in[n]\backslash\{j\}}; \xi\right),\]
 and where the expectation is taken with respect to data points $(X_1,Y_1),\dots,(X_n,Y_n),(X_{n+1},Y_{n+1})$  sampled i.i.d.\ from $P$ and an independent random seed $\xi \sim \textnormal{Unif}[0,1]$.
 \end{Definition}

\begin{figure}[t]
\centering\medskip

\fbox{\begin{tikzpicture}
\begin{scope}[xshift=0cm,yshift=0cm]
\node (beta) at (6,-0.7) {Bound on $\ell_1$-stability parameter $\beta_1(\cA,P,n)$};
\node (zero) at (0,0.1) {$0$};
\node (midpt) at (3,0.06) {$\mathcal{O}(n^{-1})$};
\node (infty) at (12,0.1) {$+\infty$};
\node (zeroline) at (0,0.5) {$\bullet$};
\node (midptline) at (3,0.5) {$\bullet$};
\draw[-latex, line width=0.5mm] (0,0.5) -> (12,0.5);
\draw [decorate, decoration = {brace, raise = 5pt, amplitude = 5pt}, line width=0.25mm] (0,0.5) --  (2.8,0.5)
	node[pos=0.5,above=10pt,black]{consistency regime};
\draw [decorate, decoration = {brace, raise = 5pt, amplitude = 5pt}, line width=0.25mm] (3.2,0.5) --  (11.6,0.5)
	node[pos=0.5,above=10pt,black]{impossibility regime};
\end{scope}
\end{tikzpicture}}

\bigskip

\fbox{\begin{tikzpicture}
\begin{scope}[xshift=0cm,yshift=0cm]
\node (beta) at (6,-0.7) {Bound on $\ell_2$-stability parameter $\beta_2(\cA,P,n)$};
\node (zero) at (0,0.1) {$0$};
\node (midpt) at (5,0.06) {$\mathcal{O}(n^{-1/2})$};
\node (infty) at (12,0.1) {$+\infty$};
\node (zeroline) at (0,0.5) {$\bullet$};
\node (midptline) at (5,0.5) {$\bullet$};
\draw[-latex, line width=0.5mm] (0,0.5) -> (12,0.5);
\draw [decorate, decoration = {brace, raise = 5pt, amplitude = 5pt}, line width=0.25mm] (0,0.5) --  (4.8,0.5)
	node[pos=0.5,above=10pt,black]{consistency regime};
\draw [decorate, decoration = {brace, raise = 5pt, amplitude = 5pt}, line width=0.25mm] (5.2,0.5) --  (11.6,0.5)
	node[pos=0.5,above=10pt,black]{impossibility regime};
\end{scope}
\end{tikzpicture}}
\caption{An illustration of the phase transition for 
performing inference on $R_{P,n}(\cA)$, relative to the $\ell_1$- or $\ell_2$-stability of the algorithm.
In the ``consistency'' regime, the questions \texttt{EvaluateAlg} and \texttt{EvaluateModel} are essentially equivalent
(as discussed in Section~\ref{sec:stability_or_consistency}), 
while in the ``impossibility'' regime, Theorem~\ref{thm:limits_evaluate_stability} establishes fundamental
limits for performing inference on the question \texttt{EvaluateAlg}. (Later on, in Theorem~\ref{thm:limits_compare_stability}, we will see that a similar phase transition
holds for the questions \texttt{CompareAlg} and \texttt{CompareModel}, as well.)}
\label{fig:phase_transition}
\end{figure}

\subsection{Stability or consistency? Defining two regimes}\label{sec:stability_or_consistency}
First, we consider a regime with a strong stability assumption,
\begin{equation}\label{eqn:beta_littleo}\beta_1(\cA,P,n) = \mathrm{o}(n^{-1}) \textnormal{ \ or \ }\beta_2(\cA,P,n) = \mathrm{o}(n^{-1/2}).\end{equation}
The work of \cite{austern2020asymptotics}, mentioned above, studies the cross-validation estimator
$\widehat{R}^{\textnormal{CV}}_n(\cA)$ under a stability condition that is roughly equivalent to 
the assumption that $\beta_2(\cA,P,n) = \mathrm{o}(n^{-1/2})$. \footnote{The related work of \cite{bayle2020cross} places a $\mathrm{o}(n^{-1/2})$ bound on a related notion of $\ell_2$-stability referred to as the loss stability (also studied earlier by  \citet{kumar2013near}). Loss stability is a strictly weaker stability condition, and
 in general is not sufficient for a consistency result (i.e., Proposition~\ref{prop:consistency_regime}).}

In fact, this level of stability can be interpreted as a \emph{consistency} assumption on the risk---as the following proposition shows.
\begin{Proposition}\label{prop:consistency_regime}
Fix any deterministic algorithm $\cA$, any distribution $P$, and any $n\geq 1$. Let $\hf_n = \cA(\{(X_i,Y_i)\}_{i\in n})$ where the data points
$(X_i,Y_i)$ are sampled i.i.d.\ from $P$. Then it holds that
\[\bbE\left[\left|R_P(\hf_n) - R_{P,n}(\cA)\right|\right] \leq 2n \beta_1(\cA,P,n).\]
and
\[\bbE\left[\left(R_P(\hf_n) - R_{P,n}(\cA)\right)^2\right]^{1/2} \leq \sqrt{n} \beta_2(\cA,P,n)\]
\end{Proposition}
\noindent Each bound is a straightforward consequence of the definition of $\ell_1$- or $\ell_2$-stability.
This result follows standard calculations and is similar to other results found in the stability literature, e.g., \cite[Lemma 4]{celisse2016stability}.
(A more technical form of this result can also be established for the case of a randomized algorithm $\cA$,
but for intuition, we restrict our attention to the deterministic case.)

Proposition \ref{prop:consistency_regime} implies that, under the strong stability assumption~\eqref{eqn:beta_littleo},
it holds that
$\bbE[(R_P(\hf_n) - R_{P,n}(\cA))^2] = \mathrm{o}(1)$.
Consequently, if we assume that this condition holds, we can estimate $R_{P,n}(\cA)$ simply by 
estimating the risk $R_P(\hf_n)$ of a \emph{single} model $\hf_n$ fitted on a training set of size $n$, which is relatively easy as long as we have a large holdout data as we have mentioned in Section \ref{sec:bounded-loss}.

At a high level, then, if we assume $\beta_q(\cA,P,n) = \mathrm{o}(n^{-1/q})$ for
either $q=1$ or $q=2$, then the distinction between questions \texttt{EvaluateAlg} and \texttt{EvaluateModel}
has essentially disappeared; estimating the risk of $\cA$ is ``easy'' for the trivial reason that it suffices
to estimate the risk $R_P(\hf_n)$ of a single fitted model produced by $\cA$.
(Of course, we would not expect this to be the optimal strategy for estimation and inference on $R_{P,n}(\cA)$.
In fact, in the setting $\beta_2(\cA,P,n) = \mathrm{o}(n^{-1/2})$ along with some additional assumptions,
 the work of \cite{austern2020asymptotics} establishes precise central limit theorem type result for the CV error which enables inference for the algorithm risk;
we do not describe the details since this is beyond the scope of the present paper.) A few common algorithms which satisfy this level of stability are $\ell_q$ regularized algorithms ($q > 1$) or soft-margin SVM \citep{bousquet2002stability,elisseeff2005stability,wibisono2009sufficient}.
On the other hand, in many practical settings, even relatively simple algorithms exhibit high variability
and do not obey a consistency type property, i.e., $\beta_q(\cA,P,n)$ may be larger than $\mathcal{O}(n^{-1/q})$ or may even
be $\mathcal{O}(1)$. This happens for instance in linear regression when $d \approx n$ or in some sparse algorithms \citep{xu2011sparse}. 

We will now present our main results, showing that outside of the consistency regime, 
there are fundamental limits on our ability to address \texttt{EvaluateAlg} with any black-box procedure. 
In particular, for $\beta_q(\cA,P,n)\gtrsim n^{-1/q}$, we will establish a hardness result, analogous to Theorem~\ref{thm:limits_evaluate},
to show that outside of the consistency regime it is impossible to perform inference on \texttt{EvaluateAlg}
with limited data. See Figure~\ref{fig:phase_transition} for an illustration of the phase transition between these two regimes.

\subsection{Limitations under stability}
We next present our main theorem in the setting of algorithmic stability.
The following theorem holds for both $q=1$ and $q=2$, i.e., for both $\ell_1$-stability and $\ell_2$-stability.

\begin{Theorem} \label{thm:limits_evaluate_stability} 
Assume that either $|\cX| = \infty$ or $|\cY| = \infty$,
and that the loss $\ell$ takes values in $[0,B]$.
Let $\hT$ be a black-box test (as in Definition \ref{def:black-box-test}),
and assume that $\hT$ satisfies the stability-constrained validity condition,
\begin{equation}\label{eqn:validity_stable}
\bbP_P( \hT(\cA,\cD_N) = 1)\leq \alpha \textnormal{ for all $\cA,P$ such that $\beta_q(\cA,P,n)\leq \gamma_q$ and $R_{P,n}(\cA)\geq \tau$,}
\end{equation}
for some stability parameter $\gamma_q \geq 2B/n^{1/q}$.
Let $\tilde\tau = \tau(1 + \frac{\alpha^{-1}-1}{N} ) $. 

Then the power of $\hT$ to detect low risk is bounded as follows:
for any $\cA$ and any $P$ with $\beta_q(\cA,P,n)\leq \gamma_q$ and $R_{P,n}(\cA)< \tau$,
\[
	\bbP_P( \hT(\cA, \cD_N)  = 1) \leq \left[\alpha\left(1 + \frac{ \tilde\tau - R_{P,n}(\cA) }{R_{P}^{\max} - \tilde\tau } \right)^{N/n} \right] \wedge 1,
\]
if we assume $\tilde\tau < R_P^{\max}$ so that the denominator is positive.
\end{Theorem}

Theorem~\ref{thm:limits_evaluate_stability} shows that, when we are outside the consistency regime (i.e., 
when the bound on $\ell_q$ stability $\gamma_q$ is at least as large as $\mathcal{O}(n^{-1/q})$), 
the problem of estimating $R_{P,n}(\cA)$ is just as hard as if we had no stability assumption at all (as in Theorem~\ref{thm:limits_evaluate}). Thus, there is no intermediate regime on our ability
to answer the question \texttt{EvaluateAlg}: instead, we have a phase transition
between the consistency regime, $\gamma_q = \mathrm{o}(n^{-1/q})$,  and the hardness results of our theorems, when $\gamma_q\gtrsim n^{-1/q}$.
As illustrated in Figure~\ref{fig:phase_transition}, we must either assume that we are in the consistency regime (in 
which case, it is essentially equivalent to simply answer the easier question \texttt{EvaluateModel}),
or without this assumption, it is impossible to have high power for testing \texttt{EvaluateAlg}. {Theorem~\ref{thm:limits_evaluate_stability} also reveals that the the stability assumption $\beta_2(\cA,P,n) = \mathrm{o}(n^{-1/2})$ assumed in \cite{austern2020asymptotics} for inference of $R_{P,n}(\cA)$ via CV error is necessary in general.}

\begin{Remark}\label{rem:comparison}
	As we have mentioned before, our results are built on the black-box testing framework established in \cite{kim2021black}. The focus of their work (i.e., establishing the hardness of testing algorithmic stability) is different than the aim of inference on algorithmic risk considered here; however, some of the proof techniques are related, particularly in the constructions of counterexamples. Here we briefly comment on some of the ways that this work extends the technical tools developed by \cite{kim2021black}. First, our work allows for more general X and Y settings (see some discussions in Remark \ref{rem:cardinality-assumption} and Appendix B.1) and and it paves the way for considering the limits of tests with computational budget constraints in settings when $|\cX|$ and $|\cY|$ are finite. Partial progress has been made in \cite{luo2024algorithmic}. Second, compared to the existing framework, we also considered relaxing the assumption-free validity in studying \texttt{EvaluateAlg} by exploring how an algorithm stability assumption helps in \texttt{EvaluateAlg}. The proof of the negative result in Theorem \ref{thm:limits_evaluate_stability} also goes beyond the original framework in \cite{kim2021black}. 
\end{Remark}

\section{Comparing multiple algorithms} \label{sec:results-comparison}
In many practical settings, we may have multiple options for which algorithm to use---for example, different regression methods,
different architectures for a neural network, different choices of a tuning parameter, etc.
In this scenario, we therefore need to compare the performance of multiple algorithms, to determine which is best
for a particular application.
In this section, we examine the question of comparing the performance of two algorithms, $\cA_0$ and $\cA_1$,
to build an understanding of the challenges of algorithm comparison.

\subsection{Problem formulation: comparing two algorithms}
We first recall the questions \texttt{CompareAlg} and \texttt{CompareModel} introduced in Section~\ref{sec:intro}.
Given two algorithms $\cA_0,\cA_1$, we would like to determine which one is a better match for a particular data analysis task:
\begin{equation*}
	\texttt{CompareAlg}\textnormal{: which algorithm, } \cA_0 \textnormal{ or }\cA_1 \textnormal{, performs better on data drawn from } P?
\end{equation*}

Of course, as for the problem of evaluating a single algorithm, we need to make a distinction between comparing $\cA_0,\cA_1$,
versus comparing the fitted models produced by these algorithms: 
\begin{equation*}
	\texttt{CompareModel}\textnormal{: which fitted model, } \hf_0 \textnormal{ or }\hf_1 \textnormal{, performs better on data drawn from } P?
\end{equation*}

To formalize these questions, we again need to consider our notion of risk. 
Given two models $f_0,f_1:\cX\rightarrow \hcY$, 
we will define the quantity
\[\Delta_P(f_0,f_1) = \bbE_P\left[ \psi( f_0(X), f_1(X),Y)\right],\]
where $\psi: \hcY\times \hcY\times\cY\rightarrow \bbR$ is a ``comparison function''. For example,
we might choose to  take a difference in losses,
\begin{equation} \label{def:psi-difference-comp}
	\psi(y',y'',y) =  \ell(y',y)-\ell(y'',y)
\end{equation}
so that $\Delta_P(f_0,f_1) = R_P(f_0) - R_P(f_1)$ simply measures the difference in risk for the two models $f_0,f_1$. Alternatively, we might choose 
\begin{equation}\label{def:psi-probability-comp}
	\psi(y',y'',y) = \indi_{\ell(y',y)> \ell(y'',y)}- \indi_{\ell(y',y)<\ell(y'',y)},
\end{equation}
so that $\Delta_P(f_0,f_1)$
compares the frequency of how often one function has a smaller loss than the other.
From this point on, we will assume implicitly that $\psi$ satisfies an antisymmetry condition,
\[\psi(y',y'',y) = - \psi(y'',y',y), \textnormal{ for all $y\in\cY$, $y',y''\in\hcY$},\] 
so that comparing $f_0$ against $f_1$  is equivalent (up to sign) to comparing $f_1$ against $f_0$.
We will take the convention that a positive value of $\psi(y',y'',y)$ indicates that $y''$ is better than $y'$, as a prediction for $y$ (as in the examples above);
in other words, $\Delta_P(f_0,f_1) > 0$ indicates that $f_1$ is better than $f_0$, for the problem of modeling data from distribution $P$.

Now, given algorithms $\cA_0,\cA_1$ for modeling our data, 
we are ready to formalize our comparison questions, \texttt{CompareAlg} and \texttt{CompareModel}.
\begin{equation*}
	\texttt{CompareAlg}\textnormal{: what is } \Delta_{P,n}(\cA_0,\cA_1)? \quad \textnormal{ versus }\quad  \texttt{CompareModel}\textnormal{: what is } \Delta_P(\hf_{0,n},\hf_{1,n})?
\end{equation*}

The question \texttt{CompareModel} compares the fitted models 
$\hf_{0,n},\hf_{1,n}$ obtained by training these algorithms on a particular data set of size $n$, given by $\hf_{l,n} = \cA_l\big(\{(X_i,Y_i)\}_{i \in [n]};\xi\big)$ for each $l=0,1$.\footnotemark\, In contrast,
 the question \texttt{CompareAlg} takes an expected value over the training set,
\[\Delta_{P,n}(\cA_0,\cA_1) = \bbE_P\left[\Delta_P(\hf_{0,n},\hf_{1,n})\right],\]
with expected value taken with respect to the distribution of both the training data $\{(X_i,Y_i)\}_{i \in [n]}$ and also the random seed $\xi\sim\textnormal{Unif}[0,1]$. 

\footnotetext{In this section we take the convention that the two algorithms, $\cA_0$ and $\cA_1$, are run with the same random seed $\xi$.
The same types of results will hold regardless of this choice, with appropriate modifications to the definitions.}

Similarly to the hypothesis test~\eqref{eqn:hypothesis_test} defined
for the problem of algorithm evaluation, we now consider addressing \texttt{CompareAlg} via performing a hypothesis testing problem regarding whether $\Delta_{P,n}(\cA_0,\cA_1)$ lies below zero or above zero,
\begin{equation}\label{eqn:hypothesis_test_comp}H_0:\,  \Delta_{P,n}(\cA_0,\cA_1) \leq 0 \textnormal{ \ versus \ } H_1:\,  \Delta_{P,n}(\cA_0,\cA_1) > 0,\end{equation}
Rejecting the null hypothesis means that there is enough evidence to conclude that $\cA_1$ is better than $\cA_0$.

\subsection{Defining a black-box test}
Similar to Definition \ref{def:black-box-test}, we can also define a black-box test in the algorithm comparison setting
\begin{Definition}[Black-box test for algorithm comparison] \label{def:black_box_test_comp}
	Let $\mathfrak{D} = \cup_{m \geq 0} (\cX \times \cY)^m$ denotes the space of data sets of any size. Consider any test $\hT$ that takes as input two algorithms $\cA_0$, $\cA_1$, a dataset $\cD\in \mathfrak{D}$, and returns a (possibly randomized) output $\hT(\cA_0, \cA_1,\cD)\in\{0,1\}$. We say that $\hT$ is a black-box test if it can be defined as follows: for some functions $g^{(1)},g^{(2)},\dots,g$, and for random seeds, $\zeta^{(1)},\zeta^{(2)},\dots,\zeta\iidsim  \textnormal{Unif}[0,1]$, 
		 \begin{enumerate}
	 	\item At stage $r =1$, generate a new dataset and a new random seed,
		\[(\cD^{(1)},\xi^{(1)} ) = g^{(1)} (\cD, \zeta^{(1)} ),\]
		and compute the fitted models $\hf_0^{(1)} = \cA_0( \cD^{(1)}; \xi^{(1)} )$ and $\hf_1^{(1)} = \cA_1( \cD^{(1)}; \xi^{(1)} )$. 
	 	\item For each stage $r = 2,3,\ldots$, generate a new dataset and a new random seed,
		\[(\cD^{(r)},\xi^{(r)} ) = g^{(r)} \big(\cD, (\cD^{(s)})_{1\leq s<r}, (\hf_l^{(s)})_{1\leq s<r;\, l=0,1}, (\zeta^{(s)})_{1\leq s\leq r} , (\xi^{(s)})_{1\leq s<r}\big),\]
		and compute the fitted models
		$\hf_0^{(r)} = \cA_0( \cD^{(r)}; \xi^{(r)} )$ and $\hf_1^{(r)} = \cA_1( \cD^{(r)}; \xi^{(r)} )$.
	 	\item Repeat until some stopping criterion is reached; let $\hr$ denote the (data-dependent) total number of rounds.
	 	\item Finally, return
		\[\hT(\cA_0, \cA_1,\cD) = g\big(\cD,(\cD^{(r)})_{1\leq r\leq \hr},(\hf_l^{(r)})_{1\leq r\leq \hr;\, l=0,1},(\zeta^{(r)})_{1\leq r\leq \hr},(\xi^{(r)})_{1\leq r\leq \hr},\zeta\big).\]
	 \end{enumerate}
\end{Definition}

\subsection{Main results for algorithm comparison}
In this section, we present our main result on the limit of any universally valid black-box test on the testing problem \eqref{eqn:hypothesis_test_comp}. Again, we would like to provide an answer to \texttt{CompareAlg} that is uniformly valid over all possible $\cA_0, \cA_1$ and $P$.
Specifically, fixing a error tolerance level $\alpha \in (0,1)$ and an evaluation sample size $n$, an available dataset $\cD_N = \{(X_i,Y_i)\}_{i\in [N]}\iidsim P$ of size $N$, we will consider tests $\hT = \hT(\cA_0, \cA_1,\cD_N)$ 
that satisfy 
\begin{equation} \label{eqn:validity_comp}
	\begin{split}
		\bbP_P \left(\hT(\cD_N, \cA_0, \cA_1) = 1\right) \leq \alpha \text{ for any } \cA_0, \cA_1, P \text{ such that } \Delta_{P,n}(\cA_0, \cA_1) \leq 0.
	\end{split}
\end{equation} 

We will restrict our attention to comparison functions $\psi$ that are bounded---the problem of testing \texttt{CompareAlg} is not meaningful in a setting where the comparison function $\psi$ is unbounded (analogous to what was shown in Theorem~\ref{thm:limits_evaluate_unbounded} for the problem \texttt{EvaluateAlg}).
In order to state our results, we first introduce an analog of $R_P^{\max}$ in this setting. For a distribution $P$ on $\cX \times \cY$, let 
\begin{equation*} 
	\begin{split}
	 \Delta_P^{\max} &= \sup_{f_0, f_1} \Delta_P(f_0, f_1),
	\end{split}
\end{equation*} the largest possible comparison risk of any two functions $f_0$ and $f_1$ for  data drawn from $P$. We are now ready to state the result.
\begin{Theorem}\label{thm:limits_compare}
Assume that either $|\cX| = \infty$ or $|\cY| = \infty$,
and that the comparison function $\psi$ takes values in $[-B, B]$. Let $\hT$ be a black-box test (as in Definition \ref{def:black_box_test_comp}),
and assume that $\hT$ satisfies the assumption-free validity condition~\eqref{eqn:validity_comp}.

Then the power of $\hT$ to compare algorithms is bounded as follows:
for any $\cA_0, \cA_1$ and any $P$ with $\Delta_{P,n}(\cA_0,\cA_1) > 0$, 
\begin{equation} \label{eqn:alg-comp-power-upper-bound}
	\bbP_P \left( \hT(\cA_0, \cA_1,\cD_N) =1 \right) \leq \left[\alpha\left( 1 +\frac{ \Delta_{P,n}(\cA_0, \cA_1) + \frac{B(\alpha^{-1}-1)}{N}}{\Delta_P^{\max} - \frac{B(\alpha^{-1}-1)}{N}} \right)^{N/n} \right] \wedge 1,
\end{equation} provided that $\Delta_P^{\max} > \frac{B(\alpha^{-1}-1)}{N}$ so that the denominator is positive.
\end{Theorem}
\noindent Since the term $\frac{B(\alpha^{-1}-1)}{N}$ is negligible as long as $N$ is large, the power is effectively bounded as $\alpha(1 + \frac{\Delta_{P,n}(\cA_0, \cA_1)}{\Delta_P^{\max}})^{N/n}$. In particular, since $\Delta_P^{\max}\geq \Delta_{P,n}(\cA_0, \cA_1)$ by definition, this means that power can only be a constant multiple of $\alpha$ if $N\propto n$. Thus, similar to the interpretation for Theorem \ref{thm:limits_evaluate}, Theorem \ref{thm:limits_compare} suggests that when either $\Delta_{P,n}(\cA_0, \cA_1)$ is close to zero or $N\propto n$, then the power of any universally valid black-box test against any alternative is fairly low. 

For the problem of algorithm evaluation, in the special case of a zero-one loss, we saw in Section~\ref{sec:matching_binomial} that the Binomial test proposed in~\eqref{eqn:binomial-test} 
provided a concrete example of a universally valid test that achieves a power nearly as high as the upper bound given in Theorem~\ref{thm:limits_evaluate}.
Analogously, for the problem of algorithm comparison, it is again possible to construct a similar test (in this case, for special cases such as the comparison
function $\psi$ given in~\eqref{def:psi-probability-comp}) to show that the upper bound on power in Theorem~\ref{thm:limits_compare} is approximately tight.
Since the intuition and construction are similar to the case of algorithm comparison, we do not give the details in this work.

\subsection{Algorithm comparison with a stability assumption}
In this section, we study the role of stability in algorithm comparison: does the problem \texttt{CompareAlg} 
become easier if we assume that $\cA_0$ and $\cA_1$ are both stable algorithms?
As for algorithm evaluation, here for the problem of algorithm comparison, we will draw the same conclusion: aside from extremely high stability  (i.e., the consistency regime), adding a stability assumption does not alter the hardness result for testing \texttt{CompareAlg}.

The following result considers a special case for the comparison function $\psi$: we assume that $\psi$ is simply the difference in losses,
as in~\eqref{def:psi-difference-comp}. 
This theorem holds for both $q=1$ and $q=2$, i.e., for both $\ell_1$-stability and $\ell_2$-stability.
\begin{Theorem} \label{thm:limits_compare_stability}
	Assume that either $|\cX| = \infty$ or $|\cY| = \infty$, 
and that the comparison function is given by $\psi(y',y'',y) = \ell(y',y) - \ell(y'',y)$, where the loss $\ell$ takes values in $[0, B]$. Let $\hT$ be a black-box test (as in Definition \ref{def:black_box_test_comp}),
and assume that $\hT$ satisfies the stability-constrained validity condition,
\begin{equation}\label{eqn:validity_stability-comp}
\begin{split}
	\bbP_P \left( \hT(\cA_0, \cA_1,\cD_N) = 1 \right)\leq \alpha \textnormal{ for all $\cA_0, \cA_1,P$ such that}\\
	 \textnormal{ $\beta_q(\cA_l,P,n)\leq \gamma_q$ for each $l=0,1$ and 
$\Delta_{n,P}(\cA_0, \cA_1) \leq  0$},
\end{split}
\end{equation}  for some stability parameter $\gamma_q \geq 2B/n^{1/q}$.

Then the power of $\hT$ to compare algorithms is bounded as follows:
for any $\cA_0, \cA_1$ and any $P$ with $\beta_q(\cA_l,P,n)\leq \gamma_q$ for $l=0,1$ and $\Delta_{P,n}(\cA_0,\cA_1)> 0$, 
\[
	\bbP_P \left( \hT(\cA_0, \cA_1,\cD_N) =1 \right) \leq \left[\alpha\left( 1 +\frac{ \Delta_{P,n}(\cA_0, \cA_1) +  \frac{B(\alpha^{-1}-1)}{N}}{\Delta_P^{\max} -  \frac{B(\alpha^{-1}-1)}{N}} \right)^{N/n} \right] \wedge 1,
\] provided that $\Delta_P^{\max} > \frac{B(\alpha^{-1}-1)}{N}$ so that the denominator is positive.
\end{Theorem}

The bound on power is the same as that in Theorem~\ref{thm:limits_compare}---that is, unless the bound $\gamma_q$ on the $\ell_q$-stability
of the two algorithms is low enough for us to be in the consistency regime, the problem of comparing algorithms suffers from the same hardness result as before.

\subsubsection{An alternative notion of stability}
In Theorem~\ref{thm:limits_compare_stability}, we considered a particular notion of stability: we assumed that \emph{each} algorithm, $\cA_0$ and $\cA_1$, 
is stable, relative to the loss function $\ell$. However, this framework only allows us to handle a particular comparison function $\psi$, namely, a difference of losses as in~\eqref{def:psi-difference-comp}.

An alternative way to incorporate stability, which will allow us to work with an arbitrary comparison function $\psi$, is to define stability directly
on the \emph{pair} of algorithms $\cA_0,\cA_1$. This type of condition was also considered by \citet[Section 4]{bayle2020cross}. 

We define the $\ell_q$-stability of the tuple $(\cA_0, \cA_1, P,n)$ as
 	\begin{multline}\label{eqn:define_stability_psi}
 		\beta_q(\cA_0, \cA_1, P,n)\\:=\max_{j=1,\ldots,n}\bbE\left[ \left| \psi(\hf_{n0}(X_{n+1}), \hf_{n1}(X_{n+1}),Y_{n+1}) - \psi(\hf_{n0}^{- j}(X_{n+1}), \hf_{n1}^{- j}(X_{n+1}),Y_{n+1})\right|^q \right] ^{1/q}.
 	\end{multline} 	 where
\[ 			\hf_{l,n} = \cA_l(\{(X_i,Y_i)\}_{i\in[n]}; \xi),\quad \hf^{- j}_{l,n} = \cA_l(\{(X_i,Y_i)\}_{i\in[n]\backslash\{j\}}; \xi), \quad l = 0,1.
 \]
  Notice that if $\psi$ is a difference of losses, as in~\eqref{def:psi-difference-comp}, then this condition is satisfied if $\cA_0,\cA_1$ are each stable
 with respect to loss $\ell$---indeed we have
 \[\beta_q(\cA_0, \cA_1, P,n)  \leq  \beta_q(\cA_0, P,n) + \beta_q(\cA_1, P,n).\]
However, it is possible to have $\beta_q(\cA_0, \cA_1, P,n)$ small even if the individual algorithms are not stable
 (if the comparison function serves to cancel noise terms, say).

We now state a hardness result under this alternative notion of stability. Again, this theorem holds for either $q=1$ or $q=2$ ($\ell_1$- or $\ell_2$-stability).

\begin{Theorem}\label{thm:limits_compare_stability_alt}
Assume that either $|\cX| = \infty$ or $|\cY| = \infty$,
and that the comparison function $\psi$ takes values in $[-B, B]$. Let $\hT$ be a black-box test (as in Definition \ref{def:black_box_test_comp}),
and assume that $\hT$ satisfies the stability-constrained validity condition,
\begin{equation}\label{eqn:validity_stability-comp2}
\begin{split}
	\bbP_P\left(  \hT(\cA_0, \cA_1,\cD_N) = 1\right)\leq \alpha \textnormal{ for any $\cA_0, \cA_1,P$ such that}\\
	 \textnormal{ $\beta_q(\cA_0, \cA_1,P,n)\leq \gamma_q$ and 
$\Delta_{n,P}(\cA_0, \cA_1) \leq  0$},
\end{split}
\end{equation}  for some stability parameter $\gamma_q \geq 4B/n^{1/q}$.

Then the power of $\hT$ to compare algorithms is bounded as follows:
for any $\cA_0, \cA_1$ and any $P$ with $\beta_q(\cA_0, \cA_1,P,n)\leq \gamma_q$ and $\Delta_{P,n}(\cA_0,\cA_1) > 0$, 
\[
	\bbP_P \left( \hT(\cA_0, \cA_1,\cD_N) =1 \right) \leq \left[\alpha\left( 1 +\frac{ \Delta_{P,n}(\cA_0, \cA_1) +  \frac{B(\alpha^{-1}-1)}{N}}{\Delta_P^{\max} -  \frac{B(\alpha^{-1}-1)}{N}} \right)^{N/n} \right] \wedge 1,
\] provided that $\Delta_P^{\max} > \frac{B(\alpha^{-1}-1)}{N}$ so that the denominator is positive.
\end{Theorem}
\noindent Again, we have the same bound on power as before, even under an assumption of stability. 

\subsection{Relating algorithm comparison and algorithm evaluation} \label{sec:algorithm-comparison-as-evaluation}
In this paper, we have considered two different problems regarding algorithms: comparing the risk of a single algorithm $\cA$ against a fixed threshold $\tau$, as in
\texttt{EvaluateAlg}, or comparing the risks of two different algorithms $\cA_0,\cA_1$ against each other, as in \texttt{CompareAlg}.
While these problems are clearly related, the hypotheses being tested appear to be distinct, but in this section
we will show that the two testing problems are actually the same---in particular, the hypothesis test~\eqref{eqn:hypothesis_test_comp}
for algorithm comparison can be viewed as a special case of the test~\eqref{eqn:hypothesis_test} for algorithm evaluation.

Specifically, if each algorithm $\cA_0,\cA_1$ returns fitted functions that are maps from $\cX$ to $\hcY$, we will now consider
the \emph{paired} algorithm $\tilde\cA = (\cA_0,\cA_1)$ that returns a function mapping from $\cX$ to $\hcY\times\hcY$. Namely,
for a data set $\cD$ and random seed $\xi\in[0,1]$, we define
$\tilde\cA(\cD;\xi) = \tilde f$, where $\tilde{f} : \cX \rightarrow \hcY\times\hcY$ is the map given by
\begin{equation} \label{eq:def-tilde-f}
	\tilde{f}(x) = (\hf_0(x), \hf_1(x)) \textnormal{ \ where \ } \hf_l = \cA_l(\cD;\xi), \ l=0,1.
\end{equation}
Next, we define 
\[\begin{array}{cccc}\tilde\ell: & (\hcY\times \hcY) \times \cY& \rightarrow &\bbR_+\\ 
& \rotatebox[origin=c]{90}{$\in$} & &  \rotatebox[origin=c]{90}{$\in$}\\
& \big( (y',y''), y\big) & \mapsto & B -  \psi(y',y'',y)
\end{array}\]
as our new loss function. In particular, note that $\tilde\ell$ takes values in $[0,2B]$, since $\psi$ takes values in $[-B,B]$.

Now we will see how the risk of the new algorithm $\tilde\cA$ relates to the comparison problem for the pair $\cA_0,\cA_1$.
 Let $\tilde R_{P,n}$ denote the risk of $\tilde{\cA}$ with respect to the new loss $\tilde\ell$, for data sampled from $P$.
For a fixed function $f: \cX \rightarrow \hcY\times\hcY$, let $f_0$ and $f_1$ be the first and second components of $f$, i.e.,
$f(x) = (f_0(x),f_1(x))$. We have
\[\tilde R_P(f) = \bbE_P[ \tilde\ell(f(X),Y)]
= \bbE_P[ B - \psi\big(f_0(X),f_1(X), Y)] = B -  \Delta_P(f_0,f_1).\]
Therefore, defining $\tilde f$, $\hf_0$, and $\hf_1$ as in \eqref{eq:def-tilde-f},
\[\tilde R_{P,n}(\tilde\cA) = \bbE[ \tilde R_P(\tilde f)] = \bbE[  B - \Delta_P(\hf_0,\hf_1) ] = B -  \Delta_{P,n}(\cA_0,\cA_1).\]
In other words, the comparison risk $\Delta_{P,n}(\cA_0,\cA_1)$ for the pair $\cA_0,\cA_1$
is equivalent (up to a transformation) to the evaluation risk $\tilde R_{P,n}(\tilde\cA)$. This means that, by replacing the space $\hcY$ with $\hcY\times \hcY$
(and $B$ with $2B$),
we can rewrite the algorithm comparison problem as an algorithm evaluation problem: 
the testing problem \eqref{eqn:hypothesis_test_comp} for \texttt{CompareAlg} is the same as testing
\begin{equation*}
	H_0: R_{P,n}(\tilde{\cA}) \geq  B  \quad \text{v.s.} \quad H_1: R_{P,n}(\tilde{\cA}) < B,
\end{equation*}
which is equivalent to the hypothesis test~\eqref{eqn:hypothesis_test} for \texttt{EvaluateAlg} if we take threshold $\tau = B$.
Indeed, as we will see in Appendix C.3, 
the result in Theorem \ref{thm:limits_compare_stability_alt} can be easily obtained from the proof of Theorem \ref{thm:limits_evaluate_stability} due to this correspondence.

\begin{Remark} \label{rem:model-comparison-unbounded}
	When the range of the comparison function $\psi$ is unbounded, i.e., $B = \infty$, then similar to the discussion in Section \ref{sec:unbounded-loss}, we would also expect that assumption-free \texttt{CompareModel} will be impossible. But luckily that in \texttt{CompareModel} or \texttt{CompareAlg}, we have the freedom to choose the comparison function $\psi$. If we choose $\psi$ to be the one in \eqref{def:psi-probability-comp}, which compares the frequency on which model or algorithm is better, then we have $B = 1$. In that case, we would still have that \texttt{CompareAlg} is much harder than \texttt{CompareModel}.
\end{Remark}

\section{Discussion} \label{sec:discussion}
In this paper, we consider the tasks of algorithm/model evaluation and algorithm/model comparison.
Our main results characterize the difference between these questions and the limitations of black-box tests to perform inference on algorithmic questions, showing that
it is impossible for a universally valid test to
have high power under any alternative hypothesis, unless the number of available data is far higher than the target sample size. 
Our results also show that an assumption of algorithmic stability is not sufficient to circumvent this hardness result, unless in the consistency regime.

Assessing the quality of models and modeling algorithms  are common tasks in many fields, we believe that our results are
useful in clarifying the gap between questions that assess an algorithm (\texttt{EvaluateAlg} and \texttt{CompareAlg}), versus
questions that assess a particular fitted model produced by the algorithm (\texttt{EvaluateModel} and \texttt{CompareModel}).
While our results establish the hardness of these first two problems, even under a certain stability assumption,
an important open question for further research is whether there are alternative mild assumptions (on the algorithm and/or on the data)
that would enable powerful inference on \texttt{EvaluateAlg} and \texttt{CompareAlg} even under sample size constraints. There is much room to explore in this direction; here we mention two lines of work, which might provide us with some inspiration. These two lines work aim to achieve distribution-free inference for the algorithm risk in slightly different settings. The first line is conformal risk control. Translate conformal risk control in our context, it assumes we have a class of algorithms parameterized by $\lambda \in \bbR$, and the loss function is monotone in $\lambda$. Then the main result of conformal risk control \citep{angelopoulosconformal} shows that we can {\it select} one algorithm, i.e., $\hat{\lambda}$, based on our data so that the selected algorithm has risk control without any assumption on the data distribution. Note that compared to our target, conformal risk control relaxes the goal by finding one particular algorithm whose risk is controlled, and it has some requirement on the class of algorithms. The second line of work is about distribution-free excess risk control. Since in many scenarios, the algorithm risk is relative to the problem difficulty, it is natural to compare the target predictor $\hf_n$ with the best possible risk achievable via some reference class of functions, say $\mathcal{F}$. For some special $\mathcal{F}$, it has been shown in \cite{mourtada2022distribution,mourtada2022improper} that we can achieve almost distribution-free upper bound for $R_{P,n}(\cA) - \inf_{f \in \mathcal{F}} R_P(f)$---the excess risk of the algorithm $\cA$ with respect to $\mathcal{F}$. Compared to our setting, these works choose to relax the target by considering one algorithm, and the inference target is changed to the excess risk.

\subsection*{Acknowledgements}
We thank the Editor, the Associated Editor, and two anonymous referees for their helpful suggestions, which helped improve the presentation and quality of this paper. We especially thank one of the reviewers for pointing out the hardness of \texttt{EvaluateModel} when the loss is unbounded, which motivates the discussion in Section \ref{sec:unbounded-loss}. The authors would also like to thank Lester Mackey for helpful comments on the paper. R.F.B.\ was supported by the Office of Naval Research via grant N00014-20-1-2337
and by the National Science Foundation via grant DMS-2023109. 

\bibliographystyle{apalike}
\bibliography{reference.bib}

\appendix

\section{Inference for Median Risk}\label{app:median-risk}
Let us define a median version of the same testing problem as in \eqref{eqn:hypothesis_test} as:
\begin{equation}\label{eqn:hypothesis_test-median}
	\texttt{EvaluateAlg}: H_0:\, \textnormal{median}(\ell(\cA(\cD_n;\xi)(X), Y) ) \geq \tau \textnormal{ \ versus \ } H_1: \,\textnormal{median}(\ell(\cA(\cD_n;\xi)(X), Y) ) <\tau,
\end{equation}where $\ell(\cA(\cD_n;\xi)(X), Y)$ means that we fit the algorithm $\cA$ on $n$ i.i.d. data $\cD_n$ drawn from $P$ with random seed $\xi$ and then evaluate the loss of the fitted model on the new independent data $(X,Y) \sim P$. Similarly, we can consider the model risk testing problem as
\begin{equation}\label{eqn:hypothesis_test2-median}
	\texttt{EvaluateModel}: H_0:\, \textnormal{median}(\ell(\hf_n(X), Y)) \geq \tau \textnormal{ \ versus \ } H_1: \,\textnormal{median}(\ell(\hf_n(X), Y)) <\tau,
\end{equation} where $\hf_n$ is a fixed model and the only randomness in computing the median is $(X,Y)$. Then, even when the loss function $\ell$ is unbounded, we can verify that solving \eqref{eqn:hypothesis_test2-median} would be much easier than \eqref{eqn:hypothesis_test-median}. To see why this is the case, let us denote $\widetilde{\ell}(\hf_n(X), Y) = \indi( \ell(\hf_n(X), Y) \geq \tau )$, $\widetilde{R}_P(\hf_n) = \bbE_P[\widetilde{\ell}(\hf_n(X), Y)| \hf_n  ]$ and $\widetilde{R}_{P,n}(\cA) = \bbE[\widetilde{R}_P(\hf_n)]$ where $\hf_n =\cA(\cD_n; \xi)$. Then we note the testing problems in \eqref{eqn:hypothesis_test-median} and \eqref{eqn:hypothesis_test2-median} are the same as
\begin{equation}\label{eqn:hypothesis_test-median3}
\begin{split}
	\texttt{EvaluateAlg}: H_0:\, \widetilde{R}_{P,n}(\cA) \geq 0.5 \textnormal{ \ versus \ } H_1: \,\widetilde{R}_{P,n}(\cA) < 0.5,\\
	\texttt{EvaluateModel}: H_0:\, \widetilde{R}_P(\hf_n) \geq 0.5 \textnormal{ \ versus \ } H_1: \,\widetilde{R}_P(\hf_n) < 0.5.
\end{split}
\end{equation}
For the new testing problem in \eqref{eqn:hypothesis_test-median3}, the essential loss function is a $0$-$1$ loss, then $\texttt{EvaluateModel}$ is easy as long as $N - n$ is large, while the hardness result for \texttt{EvaluateAlg} in Theorem \ref{thm:limits_evaluate} would still apply.

\section{Proofs: hardness results for algorithm evaluation}

\subsection{Key lemma: the role of infinite cardinality}\label{app:lemma_infinite_cardinality}
Before proving our hardness results for algorithm evaluation, we present a lemma that will play a key role in proving all our main results.
This lemma illustrates the reason why we need to assume that $\cX\times\cY$ has infinite cardinality.
\begin{Lemma}\label{lem:find_rare_point}
Assume that either $|\cX| = \infty$ or $|\cY| = \infty$.
Let $\cA$ be any algorithm, let $P$ be any distribution on $\cX\times\cY$, and let $\hT$ be any black-box test (as in Definition~\ref{def:black-box-test}).
Let random variables $\cD_N, \cD^{(1)},\dots,\cD^{(\hr)}$ denote the data sets 
that appear when sampling data from $P$ and running $\hT$ on algorithm $\cA$.
Then, for any $\epsilon>0$, there are only finitely many points $(x,y)\in\cX\times\cY$ such that\footnote{Here we interpret $(x,y)\in\cD$, for a data set $\cD \in(\cX\times\cY)^M$, to mean
that at least one of the $M$ data points in $\cD$ is equal to $(x,y)$.}
\[\bbP((x,y) \in \cD_N\cup \cD^{(1)}\cup \dots \cup \cD^{(\hr)} )  \geq \epsilon.\]
The same result also holds for any algorithms $\cA_0,\cA_1$, with $\hT$ now denoting a black-box test for comparing the algorithms (as in Definition~\ref{def:black_box_test_comp}).
\end{Lemma}
\noindent  That is,  we are considering the probability of the data point $(x,y)$ appearing
anywhere in the run of the black-box test---whether in the original data $\cD_N$ sampled from $P$, 
or in the synthetic data sets $\cD^{(r)}$ generated during the test.

\begin{proof}[Proof of Lemma~\ref{lem:find_rare_point}]
First, by definition of a black-box test, $\hT$ runs for a finite number of rounds $\hr$ (note that $\hr$ is a random variable, e.g., it may
depend on the outcomes of the previous rounds, and can be arbitrarily large, but must be finite depending on $P$ and $\cA$). Find some value $C_1$ such that $\bbP(\hr\leq C_1)\geq 1- \epsilon/4$.
Next, for each $r=1,\dots,C_1$, the data set $\cD^{(r)}$ has a finite size $N_r$ (again, $N_r$ is a random variable and may be 
arbitrarily large, but must be finite depending on $P$ and $\cA$). For each $r=1,\dots,C_1$, find some $C_{2,r}$ such that $\bbP(N_r \leq C_{2,r})\geq 1- \epsilon/4C_1$. 
Let $C_2 = \max_{r=1,\dots,C_1} C_{2,r}$. 
Therefore, with probability at least $1-\epsilon/2$, $\hT$ runs for at most $C_1$ many rounds,
and each round generates a data set $\cD^{(r)}$ with at most $C_2$ many data points. 
Let $\cE$ denote this event, so that we have $\bbP(\cE)\geq 1-\epsilon/2$. On the event $\cE$,
 then, the combined data set $\cD_N\cup \cD^{(1)}\cup \dots \cup \cD^{(\hr)}$ contains at most $N+C_1C_2$ many data points.
Now consider a data point $(x,y)$ for which 
\begin{equation} \label{ineq:lemma1-ineq1}
	\bbP((x,y) \in \cD_N\cup \cD^{(1)}\cup \dots \cup \cD^{(\hr)} ) \geq \epsilon.
\end{equation}
For any such $(x,y)$, we have
\begin{equation} \label{ineq:lemma1-ineq2}
	\bbP(\textnormal{$\cE$ holds, and }(x,y) \in \cD_N\cup \cD^{(1)}\cup \dots \cup \cD^{(\hr)} )  \geq
\epsilon - \bbP(\cE^c) \geq 
 \epsilon/2 .
\end{equation}
 Suppose there are $M$ many different $(x_i, y_i)$ pairs satisfying \eqref{ineq:lemma1-ineq1}, then 
\begin{equation*}
	\begin{split}
		N + C_1C_2 &\geq \bbE[ \cD_N\cup \cD^{(1)}\cup \dots \cup \cD^{(\hr)} \cdot \indi_{\cE}]\\
		& \geq \bbE \left[ (\sum_{i=1}^M \indi_{(x_i, y_i) \in \cD_N\cup \cD^{(1)}\cup \dots \cup \cD^{(\hr)}} ) \cdot \indi_{\cE} \right]\\
		& = \sum_{i=1}^M \bbP \left( \textnormal{$\cE$ holds, and }(x_i,y_i) \in \cD_N\cup \cD^{(1)}\cup \dots \cup \cD^{(\hr)}  \right) \\
		& \overset{ \eqref{ineq:lemma1-ineq2} }\geq M \epsilon/2, 
	\end{split}
\end{equation*} so we have $M \leq \frac{N+C_1C_2}{\epsilon/2}$, i.e., the inequality \eqref{ineq:lemma1-ineq1} can only hold for at most $\frac{N+C_1C_2}{\epsilon/2}$ many data points.
\end{proof}
This result reveals how the assumption of infinite cardinality will be used---if only finitely many points $(x,y)$
have probability $\geq \epsilon$ of appearing at any point during the black-box test, then we can always find some other data point $(x',y')\in\cX\times\cY$
whose probability of appearing during the test is $<\epsilon$.
Moreover, the proof also suggests how we might
derive analogous results for \emph{finite} $\cX\times\cY$ under a computational constraint.
Specifically, if $\hT$ is only allowed to call $\cA$ a bounded number of times, and if each run of $\cA$ is only allowed to input
a data set of some bounded size, then these constraints can be used in place of the (arbitrarily large) values $C_1,C_2$ in the proof above.

\subsection{Proof of Theorem~\ref{thm:limits_evaluate}}
Theorem~\ref{thm:limits_evaluate} can be viewed as a special case of 
Theorem~\ref{thm:limits_evaluate_stability}. 
Specifically, since the loss takes values in $[0,B]$, 
it holds trivially for any $(\cA,P,n)$ that $\beta_q(\cA,P,n) \leq B$.
Thus, the result of Theorem~\ref{thm:limits_evaluate} is implied by
applying Theorem~\ref{thm:limits_evaluate_stability} with stability parameter $\gamma_q = B$.

\subsection{Proof of Theorem~\ref{thm:limits_evaluate_unbounded}}
Fix any $\cA$ and any $P$. 
Let $\epsilon,\delta>0$ be fixed and arbitrarily small constants.

Define a constant
\[C = \frac{\tau}{\left[1 - \left(1 - \delta / 2\right)^n\right]\cdot \delta/2}.\]
Since $\sup_{\hy,y}\ell(\hy,y)=\infty$, we can find some $\hy_0,y_0$ such that $\ell(\hy_0,y_0) \geq C$.
And, by Lemma~\ref{lem:find_rare_point}, we can find some $(x_*,y_*)\in\cX\times\cY$ such that
\[\bbP((x_*,y_*) \in \cD_N\cup \cD^{(1)}\cup \dots \cup \cD^{(\hr)} ) \leq  \epsilon,\]
where $\cD_N$ denotes an i.i.d.\ sample from $P$, and $\cD^{(1)},\dots,\cD^{(\hr)}$ denote
the data sets constructed when running $\hT(\cA,\cD_N)$.
Let $\cE_*$ be the event that $(x_*,y_*) \not\in \cD_N\cup \cD^{(1)}\cup \dots \cup \cD^{(\hr)}$,
so that $\bbP(\cE_*)\geq 1-\epsilon$.

Next,  we define a new distribution
\[P' = (1 - \delta ) \cdot P + \delta/2 \cdot (P_X\times \delta_{y_0}) + \delta/2 \cdot \delta_{(x_*,y_*)},\]
and  a new algorithm $\cA'$ as follows: for any $m\geq 0$, any $\cD\in(\cX\times\cY)^m$,
and any $\xi\in[0,1]$,
\[\cA'(\cD;\xi) = \begin{cases} \cA(\cD;\xi), & \textnormal{ if }(x_*,y_*)\not\in \cD,\\
f(x)\equiv \hy_0, & \textnormal{ if }(x_*,y_*)\in \cD.\end{cases}\]
That is, if $\cD$ contains the data point $(x_*,y_*)$, then $\cA'$ returns the constant function that maps any $x$ to the value $\hy_0$.
Now let $\cD_N$ be drawn i.i.d.\ from $P$, and let $\cD_N'$ be drawn i.i.d.\ from $P'$. By definition of $P'$,
we can construct a coupling between these two data sets such that $\bbP(\cD_N = \cD_N' \mid \cD_N) \geq (1 - \delta)^N$. 
Therefore, this means that
\[\bbP(\hT(\cA',\cD_N') = 1) \geq (1-\delta)^N \bbP(\hT(\cA',\cD_N) = 1) .\]

Next, observe that on the event $\cE_*$, we have $\hT(\cA',\cD_N) = \hT(\cA,\cD_N)$, almost surely.
Informally, this is because when running the black-box text $\hT$,
if we initialize at the same data set $\cD_N$, and we never observe the data point $(x_*,y_*)$,
then $\cA$ and $\cA'$ are returning the same fitted models at every iteration, and so 
all iterations of the tests $\hT(\cA,\cD_N)$ and $\hT(\cA',\cD_N)$ are identical to each other.
(A more formal proof of this type of statement can be found in the proof of \citet[Theorem 2]{kim2021black}.)
In particular, we have
\[\bbP(\hT(\cA',\cD_N) = 1) \geq \bbP(\hT(\cA,\cD_N) = 1) - \bbP(\cE_*^c) \geq  \bbP(\hT(\cA,\cD_N) = 1) - \epsilon.\]
Combining our calculations, then,
\[ \bbP(\hT(\cA,\cD_N) = 1) \leq (1-\delta)^{-N}\bbP(\hT(\cA',\cD_N') = 1) + \epsilon.\]

To complete the proof, we
 now calculate the risk of the new algorithm and distribution: writing $\cD_n'$ to denote a data set drawn i.i.d.\ from $P'$, and $(X_{n+1}',Y_{n+1}')\sim P'$
as the test point,
\begin{align*}
R_{P',n}(\cA')
& = \bbE_{P'}\left[ \ell\big( [\cA'(\cD'_n;\xi)](X_{n+1}'), Y_{n+1}'\big)\right]\\
&\geq  \bbE_{P'}\left[ \ell\big( [\cA'(\cD'_n;\xi)](X_{n+1}'), Y_{n+1}'\big)\cdot\indi_{(x_*,y_*)\in\cD'_n}\cdot\indi_{Y_{n+1}'=y_0}\right]\\
&=  \bbE_{P'}\left[ \ell(\hy_0,y_0)\cdot\indi_{(x_*,y_*)\in\cD'_n}\cdot\indi_{Y_{n+1}'=y_0}\right]\\
&\geq C\cdot   \bbE_{P'}\left[ \indi_{(x_*,y_*)\in\cD'_n}\cdot\indi_{Y_{n+1}'=y_0}\right]\\
&\geq C\cdot \left[1 - \left(1 - \delta / 2\right)^n\right] \cdot \delta/2\\
& = \tau,
\end{align*}
where the last step holds by definition of $C$.
Therefore, by validity of $\hT$~\eqref{eqn:validity}, we must have
\[\bbP(\hT(\cA',\cD_N')=1)\leq \alpha \ \Longrightarrow \ \bbP(\hT(\cA,\cD_N) =1 ) \leq (1-\delta)^{-N}\alpha + \epsilon. \]
Since $\epsilon,\delta>0$ can be taken to be arbitrarily small (while $N$ is being treated as fixed),
this proves the theorem.

\subsection{Proof of Theorem~\ref{thm:limits_evaluate_stability}}
The general proof idea is the following: for any $\cA$ and $P$ such that $\beta_q(\cA, P, n) \leq \gamma_q$ and $R_{P,n}(\cA) < \tau$, we can construct a new algorithm $\cA'$ and a new distribution $P'$ such that $\beta_q(\cA', P', n) \leq \gamma_q$ and $R_{P',n}(\cA') \geq \tau$. Moreover, the constructed $(\cA', P')$ is sufficiently similar to $(\cA, P)$ so that they are difficult to distinguish using limited $N$ i.i.d. data points. 

To make this idea precise, we begin by fixing any constant $c\in(0,1)$ with
\begin{equation}\label{eqn:assume_c}
(1-c)^n \cdot R_{P,n}(\cA)
+ \left[1-(1-c)^n\right]\cdot  R_P^{\max}  > \tilde\tau.\end{equation}
We will show that
\begin{equation}\label{eqn:prove_with_c}\bbP(\hT(\cA,\cD_N)=1)
\leq (1-c)^{-N} \alpha.\end{equation}
In particular, since condition~\eqref{eqn:assume_c} holds for any $c$ satisfying
\[ (1-c)< \left(\frac{R_P^{\max} - \tilde\tau}{R_P^{\max} - R_{P,n}(\cA)}\right)^{1/n},\]
this implies
\[\bbP(\hT(\cA,\cD_N) = 1) \leq \alpha\left(\frac{ R_{P}^{\max} - R_{P,n}(\cA) }{R_{P}^{\max} - \tilde\tau }\right)^{N/n} = \alpha\left(1 + \frac{ \tilde\tau - R_{P,n}(\cA) }{R_{P}^{\max} - \tilde\tau } \right)^{N/n} ,\]
as claimed. (Of course, $\bbP(\hT(\cA,\cD_N)=1)\leq 1$ holds trivially.)

We now prove~\eqref{eqn:prove_with_c} holds for any fixed $c$ satisfying~\eqref{eqn:assume_c}.

\subsubsection{Step 1: some preliminaries}
First, if  $1-c < (1/\alpha)^{-1/N}$, then no matter whether \eqref{eqn:assume_c} holds or not, the desired bound~\eqref{eqn:prove_with_c} holds trivially, since the right-hand side of~\eqref{eqn:prove_with_c} is $>1$ in this case.
Therefore we can restrict our attention to the nontrivial case where $1-c\geq (1/\alpha)^{-1/N}$. 
This implies that 
\[\tilde\tau = \tau\left(1 + \frac{\alpha^{-1}-1}{N} \right)  \geq  \tau\cdot (1/\alpha)^{1/N}\geq \frac{\tau}{1-c}.\]
Combining this with~\eqref{eqn:assume_c} then yields
\begin{equation}\label{eqn:assume_c_2}(1-c) \left[(1-c)^n \cdot R_{P,n}(\cA)
+ \left[1-(1-c)^n\right]\cdot  R_P^{\max}\right]
>  \tau.\end{equation}

From this point on, 
we let $\epsilon\in (0,c)$ denote an arbitrarily small constant.
By definition of $R_P^{\max}$, we can find some function $f_*:\cX\rightarrow\hcY$ such that
\begin{equation}\label{eqn:star1}R_P(f_*)\geq R_P^{\max} - \epsilon.\end{equation}
And, by Lemma~\ref{lem:find_rare_point}, we can find some $(x_*,y_*)\in\cX\times\cY$
such that
\begin{equation}\label{eqn:star2}\bbP((x_*,y_*) \in \cD_N\cup \cD^{(1)}\cup \dots \cup \cD^{(\hr)} )  \leq  \epsilon,\end{equation}
where $\cD_N$ denotes an i.i.d.\ sample from $P$, and $\cD^{(1)},\dots,\cD^{(\hr)}$ denote
the data sets constructed when running $\hT(\cA,\cD_N)$.

\subsubsection{Step 2: constructing a new algorithm and a new distribution}
We begin by explaining the intuition behind the construction of $P'$ and $\cA'$,
which is motivated by the procedure of the black-box test. As in the proof of Theorem \ref{thm:limits_evaluate_unbounded}, we want $\cA'$ to behave in the same way as $\cA$ on a typical dataset drawn from $P$, so that $\hT$ cannot easily distinguish between $\cA$ and $\cA'$---but on the other hand, we want $R_{P,n}(\cA') \geq \tau$
so that $\hT$ is constrained to have a low probability of returning a $1$ when run on $\cA'$ (and consequently, also when run on the original algorithm of interest $\cA$).

With this in mind, writing $P_*$ as the distribution of $(X,Y)\sim P$ conditional on the event $(X,Y) \neq (x_*,y_*)$, we 
can express $P$ as a mixture
\[P = (1-p_*)\cdot P_* + p_*\cdot \delta_{(x_*,y_*)},\]
where $p_* = \bbP_P((X,Y) = (x_*,y_*))$. Note that $p_*\leq \bbP((x_*,y_*)\in\cD_N)\leq \epsilon$ by~\eqref{eqn:star2} (since $N\geq 1$), and $\epsilon\leq c$ by assumption.
Now  define a new distribution,
\[P' =\frac{1-c}{1-p_*}\cdot P + \frac{c-p_*}{1-p_*}\cdot \delta_{(x_*,y_*)}  = (1-c)\cdot P_* + c\cdot \delta_{(x_*,y_*)},\]
and a new algorithm $\cA'$, defined as
\[\cA'(\cD;\xi) = \begin{cases} \cA(\cD;\xi), & \textnormal{ if }(x_*,y_*)\not\in \cD,\\
f_*, & \textnormal{ if }(x_*,y_*)\in \cD.\end{cases}\]
Exactly as in the proof of Theorem~\ref{thm:limits_evaluate_unbounded}, we can calculate
\[\bbP(\hT(\cA,\cD_N)=1)
\leq \left(\frac{1-c}{1-p_*}\right)^{-N}\bbP(\hT(\cA',\cD_N')=1) + \epsilon \leq (1-c)^{-N}\bbP(\hT(\cA',\cD_N')=1) + \epsilon,\]
where $\cD_N$, $\cD_N'$ are data sets sampled i.i.d.\ from $P$ and from $P'$, respectively.

We will show below that the new algorithm $\cA'$ and distribution $P'$ satisfy
$R_{P',n}(\cA')\geq \tau$ and $\beta_q(\cA',P',n) \leq \gamma_q$, as long as we choose $\epsilon$ to be sufficiently small.
This means that we must have
$\bbP(\hT(\cA',\cD_N')=1) \leq \alpha$, by the assumption~\eqref{eqn:validity_stable} on the validity of $\hT$ for testing
stable algorithms, and so we have
\[\bbP(\hT(\cA,\cD_N)=1)
\leq (1-c)^{-N}\alpha + \epsilon.\]
Since $\epsilon>0$ can be chosen to be arbitrarily small, then, this verifies the bound~\eqref{eqn:prove_with_c},
 which completes the proof of the theorem once we have verified the risk and stability of $(\cA',P',n)$.

\subsubsection{Step 3: verifying the risk}
In this step we verify that $R_{P',n}(\cA')\geq \tau$.
Write $\hf_n' = \cA'(\cD_n';\xi)$ where $\cD_n' =\{(X_i',Y_i')\}$ is sampled i.i.d.\ from $P'$, and let $(X_{n+1}',Y_{n+1}')\sim P'$
be an independently drawn test point. We calculate
\begin{align*}
R_{P',n}(\cA')
&=\bbE\left[\ell(\hf_n'(X_{n+1}'),Y_{n+1}')\right]   \geq \bbE\left[\ell(\hf_n'(X_{n+1}'),Y_{n+1}')\cdot \indi_{(X_{n+1}',Y_{n+1}')\neq (x_*,y_*)}\right] \\
&\geq\bbE\left[\ell(\hf_n'(X_{n+1}'),Y_{n+1}')\cdot \indi_{(x_*,y_*)\not\in \cD_n',  (X_{n+1}',Y_{n+1}')\neq (x_*,y_*)}\right]
\\&\hspace{1in}+ \bbE\left[\ell(\hf_n'(X_{n+1}'),Y_{n+1}')\cdot  \indi_{(x_*,y_*)\in \cD_n',  (X_{n+1}',Y_{n+1}')\neq (x_*,y_*)}\right]\\
&=\bbE\left[\ell(\hf_n(X_{n+1}'),Y_{n+1}')\cdot \indi_{(x_*,y_*)\not\in \cD_n',  (X_{n+1}',Y_{n+1}')\neq (x_*,y_*)}\right]
\\&\hspace{1in}+ \bbE\left[\ell(f_*(X_{n+1}'),Y_{n+1}')\cdot  \indi_{(x_*,y_*)\in \cD_n',  (X_{n+1}',Y_{n+1}')\neq (x_*,y_*)}\right],\end{align*}
where the last step holds by definition of the modified algorithm $\cA'$. 
Since the distribution $P'$ places mass $c$ on the point $(x_*,y_*)$, we can rewrite this as
\begin{align*}
R_{P',n}(\cA')
&\geq (1-c)^{n+1} \bbE\left[\ell(\hf_n(X_{n+1}'),Y_{n+1}')\mid  (x_*,y_*)\not\in \cD_n',  (X_{n+1}',Y_{n+1}')\neq (x_*,y_*)\right]\\
&\hspace{1in} + (1-c)\left[1-(1-c)^n\right] \cdot \bbE\left[\ell(f_*(X_{n+1}'),Y_{n+1}')\mid   (X_{n+1}',Y_{n+1}')\neq (x_*,y_*)\right].\end{align*}
And, since $P=(1-c)P_* + c\delta_{(x_*,y_*)}$, a data point $(X,Y)\sim P'$ has distribution $P_*$ if we condition on the event $(X,Y)\neq (x_*,y_*)$.
This means that we can write
\[\bbE\left[\ell(\hf_n(X_{n+1}'),Y_{n+1}')\mid  (x_*,y_*)\not\in \cD_n',  (X_{n+1}',Y_{n+1}')\neq (x_*,y_*)\right] = R_{P_*,n}(\cA),\]
the risk of the original algorithm $\cA$ for data drawn from $P_*$, while
\[\bbE\left[\ell(f_*(X_{n+1}'),Y_{n+1}')\mid   (X_{n+1}',Y_{n+1}')\neq (x_*,y_*)\right] = R_{P_*}(f_*),\]
the risk of the function $f_*$ for a test point drawn from $P_*$. 
Therefore,
\[R_{P',n}(\cA') \geq 
(1-c)^{n+1}\cdot R_{P_*,n}(\cA) +  (1-c)\left[1-(1-c)^n\right] \cdot R_{P_*}(f_*).\]

Now we bound these remaining risk terms.
First, since $P = (1-p_*)P_* + p_*\delta_{(x_*,y_*)}$, we
 can relate $R_{P_*,n}(\cA)$ to $R_{P,n}(\cA)$ as follows: 
 writing $\cD_n=\{(X_i,Y_i)\}_{i\in[n]}$ and $(X_{n+1},Y_{n+1})$ to denote training
and test data points sampled i.i.d.\ from $P$, and $\hf_n = \cA(\cD_n;\xi)$ as the fitted model, we have
\begin{multline*}
R_{P,n}(\cA)
 = \bbE\left[\ell(\hf_n(X_{n+1}),Y_{n+1})\right]\\
  \leq  (1-p_*)^{n+1}\bbE\left[\ell(\hf_n(X_{n+1}),Y_{n+1})\mid (x_*,y_*)\not\in\cD_n\cup\{(X_{n+1},Y_{n+1}\}\right]
  + \left[ 1 - (1-p_*)^{n+1}\right] B\\
=  (1-p_*)^{n+1} \cdot R_{P_*,n}(\cA)
+ \left[ 1 - (1-p_*)^{n+1}\right] \cdot B
\leq  R_{P_*,n}(\cA)
+ \epsilon(n+1) B ,\end{multline*}
since $p_*\leq \epsilon$ and so $ \left[ 1 - (1-p_*)^{n+1}\right]  \leq  \left[ 1 - (1-\epsilon)^{n+1}\right]  \leq (n+1)\epsilon$. Rearranging terms, 
\[R_{P_*,n}(\cA) \geq R_{P,n}(\cA) -  \epsilon (n+1) B.\] Similarly, we have
\[R_P(f_*) = (1-p_*) R_{P_*}(f_*)  +p_* \ell(f_*(x_*),y_*) \leq (1-p_*) R_{P_*}(f_*) + p_* B \leq R_{P_*}(f_*) + \epsilon B,\]
and we also know that $R_P(f_*)\geq R_P^{\max}-\epsilon$ by~\eqref{eqn:star1}. So,
\[R_{P_*}(f_*) \geq R_P^{\max}-\epsilon(1+B).\]

Combining everything, then,
\begin{multline*}R_{P',n}(\cA') \geq 
(1-c)^{n+1}\cdot \left[R_{P,n}(\cA) -  \epsilon (n+1) B\right] +  (1-c)\left[1-(1-c)^n\right] \cdot \left[ R_P^{\max} - \epsilon(1+B)\right]\\
\geq (1-c) \Big[(1-c)^n \cdot R_{P,n}(\cA)
+ \left[1-(1-c)^n\right]\cdot  R_P^{\max}\Big] - \epsilon\left( 1 + (n+1)B \right) \\> \tau - \epsilon\left( 1 + (n+1)B \right) ,\end{multline*}
by~\eqref{eqn:assume_c_2}. We therefore have
$R_{P',n}(\cA')\geq \tau$
for sufficiently small $\epsilon>0$, as desired.

\subsubsection{Step 4: verifying the stability}
In this step we verify that $\beta_q(\cA',P',n) \leq \gamma_q$.
Let $(X_1,Y_1),\dots,(X_{n+1},Y_{n+1})\iidsim P$, and let $$(X_1',Y_1'),\dots,(X_{n+1}',Y_{n+1}')\iidsim P'.$$ 
Define $\cD_n = \{(X_i,Y_i)\}_{i\in[n]}$, let $\cD_{n+1}= \{(X_i,Y_i)\}_{i\in[n+1]}$, and let 
 $\cD_n^{(-j)} = \{(X_i,Y_i)\}_{i\in[n]\backslash\{j\}}$. Define $\cD'_n$, $\cD'_{n+1}$, and $\cD'_n{}^{(-j)}$ analogously.
 Let $\hf_n = \cA(\cD_n;\xi)$ and $\hf_n^{(-j)} = \cA(\cD_n^{(-j)};\xi)$, and $\hf_n' = \cA'(\cD_n';\xi)$ and $\hf_n'{}^{(-j)} = \cA'(\cD'_n{}^{(-j)};\xi)$.

First, we work with the stability of $(\cA,P,n)$. We have
\begin{align*}
\gamma^q_q \geq  \beta^q_q(\cA,P,n)
&=   \bbE\left[ \left| \ell(\hf_n(X_{n+1}),Y_{n+1}) - \ell(\hf_n^{(-j)}(X_{n+1}),Y_{n+1})\right|^q \right] \\
&\geq   \bbE\left[\indi_{(x_*,y_*)\not\in\cD_{n+1}}  \left| \ell(\hf_n(X_{n+1}),Y_{n+1}) - \ell(\hf_n^{(-j)}(X_{n+1}),Y_{n+1})\right|^q \right] \\
&= (1-p_*)^{n+1} \cdot   \bbE_{P_*}\left[  \left| \ell(\hf_n(X_{n+1}),Y_{n+1}) - \ell(\hf_n^{(-j)}(X_{n+1}),Y_{n+1})\right|^q\right]\\
&\geq (1-\epsilon)^{n+1} \cdot   \bbE_{P_*}\left[  \left| \ell(\hf_n(X_{n+1}),Y_{n+1}) - \ell(\hf_n^{(-j)}(X_{n+1}),Y_{n+1})\right|^q\right],
\end{align*}
since conditional on   $(x_*,y_*)\not\in\cD_{n+1}$, the data points $(X_i,Y_i)$ are i.i.d.\ draws from $P_*$, and since $p_*\leq \epsilon$.
Similarly, we can calculate
\begin{align*}
 \beta^q_q(\cA',P',n)
&=   \bbE\left[ \left| \ell(\hf'_n(X'_{n+1}),Y'_{n+1}) - \ell(\hf'_n{}^{(-j)}(X'_{n+1}),Y'_{n+1})\right|^q \right] \\
&=   \bbE\left[\indi_{(x_*,y_*)\not\in\cD_{n+1}'} \left| \ell(\hf'_n(X'_{n+1}),Y'_{n+1}) - \ell(\hf'_n{}^{(-j)}(X'_{n+1}),Y'_{n+1})\right|^q \right] 
\\&\hspace{.2in}+   \bbE\left[\indi_{(x_*,y_*)\in\cD_{n+1}'} \left| \ell(\hf'_n(X'_{n+1}),Y'_{n+1}) - \ell(\hf'_n{}^{(-j)}(X'_{n+1}),Y'_{n+1})\right|^q \right] \\
&=   (1-c)^{n+1}\bbE\left[ \left| \ell(\hf'_n(X'_{n+1}),Y'_{n+1}) - \ell(\hf'_n{}^{(-j)}(X'_{n+1}),Y'_{n+1})\right|^q \ \big\vert \ (x_*,y_*)\not\in\cD_{n+1}'\right] 
\\&\hspace{.2in}+\left[1 - (1-c)^{n+1}\right]   \bbE\left[\left| \ell(\hf'_n(X'_{n+1}),Y'_{n+1}) - \ell(\hf'_n{}^{(-j)}(X'_{n+1}),Y'_{n+1})\right|^q \ \big \vert \ (x_*,y_*)\in\cD_{n+1}' \right] .
\end{align*}
Next, on the event 
$(x_*,y_*)\not\in\cD_{n+1}'$, we have $\hf'_n = \hf_n$ and $\hf'_n{}^{(-j)}=\hf_n^{(-j)}$, and the data points $(X_i',Y_i')$ are i.i.d.\ draws from $P_*$.
So, 
 \begin{multline*}\bbE\left[ \left| \ell(\hf'_n(X'_{n+1}),Y'_{n+1}) - \ell(\hf'_n{}^{(-j)}(X'_{n+1}),Y'_{n+1})\right|^q \ \big\vert \ (x_*,y_*)\not\in\cD_{n+1}'\right]  \\= 
   \bbE_{P_*}\left[  \left| \ell(\hf_n(X_{n+1}),Y_{n+1}) - \ell(\hf_n^{(-j)}(X_{n+1}),Y_{n+1})\right|^q\right]
  \leq  (1-\epsilon)^{-(n+1)} \cdot\gamma^q_q.\end{multline*}
And, if  $(x_*,y_*) \in \cD'_n{}^{(-j)}$, then $\hf'_n = \hf'_n{}^{(-j)}$, so 
\begin{multline*}\bbE\left[\left| \ell(\hf'_n(X'_{n+1}),Y'_{n+1}) - \ell(\hf'_n{}^{(-j)}(X'_{n+1}),Y'_{n+1})\right|^q \ \big \vert \ (x_*,y_*)\in\cD_{n+1}' \right]\\
\leq B^q\cdot  \bbP( (x_*,y_*) \not \in \cD'_n{}^{(-j)}\mid  (x_*,y_*)\in\cD_{n+1}' )
\leq B^q\cdot  \frac{2}{n+1},\end{multline*}
where the last step holds since data points $\{(X_i',Y_i')\}_{i\in[n+1]}$ are exchangeable, so if $(x_*,y_*)$ appears at least once in a data set of size $n+1$,
with probability at least $\frac{n-1}{n+1}$ it appears in the subset given by the $n-1$ many indices $[n]\backslash \{j\}$.
Combining these calculations,
\[
 \beta^q_q(\cA',P',n)
 \leq (1-c)^{n+1} \cdot (1-\epsilon)^{-(n+1)} \cdot\gamma^q_q+ \left[1 - (1-c)^{n+1}\right]  \cdot B^q \cdot \frac{2}{n+1} .\]
By assumption, $\gamma^q_q\geq B^q \cdot \frac{2}{n} $, so 
we have
\[
 \beta^q_q(\cA',P',n)
 \leq \gamma^q_q \cdot \left[ (1-c)^{n+1} \cdot (1-\epsilon)^{-(n+1)} + \left[1 - (1-c)^{n+1}\right]  \cdot \frac{n}{n+1}\right].\]
Since $c>0$,  we therefore have $ \beta^q_q(\cA',P',n)\leq \gamma^q_q$ for sufficiently small $\epsilon>0$, as desired.

\section{Proofs: hardness results for algorithm comparison}
\subsection{Proof of Theorem~\ref{thm:limits_compare}}
Theorem~\ref{thm:limits_compare} can be viewed as a special case of 
Theorem~\ref{thm:limits_compare_stability_alt}. 
Specifically, since the comparison function takes values in $[-B,B]$, 
it holds trivially for any $(\cA_0, \cA_1,P,n)$ that $\beta_q(\cA_0, \cA_1,P,n) \leq 2B$.
Thus, the result of Theorem~\ref{thm:limits_evaluate} is implied by
applying Theorem~\ref{thm:limits_compare_stability_alt} with stability parameter $\gamma_q = 2B$.

\subsection{Proof of Theorem \ref{thm:limits_compare_stability} }
The proof of this Theorem is very similar to the proof of Theorem~\ref{thm:limits_evaluate_stability}, with a few modifications
to accommodate the algorithm comparison setting.
We begin by fixing any constant $c$ with
\begin{equation}\label{eqn:assume_c_comparison}
(1-c) < \left( \frac{\Delta_P^{\max} - \frac{B(\alpha^{-1} - 1)}{N} }{\Delta_P^{\max} + \Delta_{P,n}(\cA_0, \cA_1)}   \right)^{1/n}
\end{equation}
We will show that
\begin{equation}\label{eqn:prove_with_c_comparison}\bbP(\hT(\cA_0,\cA_1,\cD_N)=1)
\leq (1-c)^{-N} \alpha.\end{equation}
for any such $c$, which implies that
\[\bbP(\hT(\cA_0,\cA_1,\cD_N) = 1) \leq \left[\alpha\left(1 + \frac{ \Delta_{P,n}(\cA_0, \cA_1) + \frac{B(\alpha^{-1} - 1)}{N} }{\Delta_P^{\max} - \frac{B(\alpha^{-1} - 1)}{N} } \right)^{N/n}\right]\wedge 1 ,\]
as claimed.

\subsubsection{Step 1: some preliminaries}
As in the proof of Theorem~\ref{thm:limits_evaluate_stability}, assume $1-c\geq (1/\alpha)^{-1/N}$ to avoid the trivial case.
By~\eqref{eqn:assume_c_comparison}, then
\[(1-c)^n \cdot (-\Delta_{P,n}(\cA_0,\cA_1)) + [1 - (1-c)^n] \cdot  \Delta_P^{\max}\\
 >   \frac{B(\alpha^{-1} - 1)}{N} 
\geq B \left[(1/\alpha)^{1/N}-1\right] \geq B \cdot\frac{c}{1-c}\]
and therefore,
\begin{equation}\label{eqn:assume_c_compare_2}(1-c) \left[(1-c)^n \cdot \Delta_{P,n}(\cA_0, \cA_1)
+ \left[1-(1-c)^n\right]\cdot ( - \Delta_P^{\max})\right] + cB
<  0.\end{equation}
Let $\epsilon\in (0,c)$ denote an arbitrarily small constant.
By definition of $\Delta_P^{\max}$, we can find functions $f_0,f_1:\cX\rightarrow\hcY$ such that
\begin{equation}\label{eqn:star1-2}\Delta_P(f_0,f_1)\geq \Delta_P^{\max} - \epsilon.\end{equation}
And, by Lemma~\ref{lem:find_rare_point}, we can find some $(x_*,y_*)\in\cX\times\cY$
such that
\begin{equation}\label{eqn:star2-2}\bbP((x_*,y_*) \in \cD_N\cup \cD^{(1)}\cup \dots \cup \cD^{(\hr)} )  \leq  \epsilon,\end{equation}
where $\cD_N$ denotes an i.i.d.\ sample from $P$, and $\cD^{(1)},\dots,\cD^{(\hr)}$ denote
the data sets constructed when running $\hT(\cA_0,\cA_1,\cD_N)$.

\subsubsection{Step 2: constructing new algorithms and a new distribution}
Let $P_*$, $p_*$, and $P'$ be defined exactly as in the proof of Theorem~\ref{thm:limits_evaluate_stability}.
Now we define the new algorithms $\cA_0',\cA_1'$ as follows:
\[
		\cA_0'(\cD; \xi) = \begin{cases}
			f_1, & \text{ if } (x^*, y^*) \in \cD, \\
			\cA_0(\cD; \xi), & \text{ if }  (x^*, y^*) \notin \cD.
		\end{cases}\quad\quad 
				\cA_1'(\cD; \xi) = \begin{cases}
			f_0, & \text{ if } (x^*, y^*) \in \cD, \\
			\cA_1(\cD; \xi), & \text{ if }  (x^*, y^*) \notin \cD.
		\end{cases}
\]
Exactly as in the proofs of Theorem~\ref{thm:limits_evaluate_unbounded} and Theorem~\ref{thm:limits_evaluate_stability}, we can calculate
\[\bbP(\hT(\cA_0,\cA_1,\cD_N)=1)
\leq (1-c)^{-N} \bbP(\hT(\cA_0',\cA_1',\cD_N')=1) + \epsilon,\]
where $\cD_N$, $\cD_N'$ are data sets sampled i.i.d.\ from $P$ and from $P'$, respectively.

We will show below that the new algorithms $\cA_0',\cA_1$ and distribution $P'$ satisfy
$\Delta_{P',n}(\cA_0',\cA_1')\leq 0$ and $\beta_q(\cA_l',P',n) \leq \gamma_q$ for $l=0,1$, as long as we choose $\epsilon$ to be sufficiently small.
This means that we must have
$\bbP(\hT(\cA_0',\cA_1',\cD_N')=1) \leq \alpha$, by the assumption~\eqref{eqn:validity_stability-comp} on the validity of $\hT$ for comparing
stable algorithms, which completes the proof exactly as for Theorem~\ref{thm:limits_evaluate_stability}.

\subsubsection{Step 3: verifying the risk}
In this step we verify that $\Delta_{P',n}(\cA_0',\cA_1')\leq 0$.
Write $\hf_{l,n}' = \cA_l'(\cD_n';\xi)$ for each $l=0,1$, where $\cD_n' =\{(X_i',Y_i')\}$ is sampled i.i.d.\ from $P'$, and let $(X_{n+1}',Y_{n+1}')\sim P'$
be an independently drawn test point. Recalling that
$\Delta_{P',n}(\cA_0',\cA_1') = R_{P',n}(\cA_0') - R_{P',n}(\cA_1')$
by assumption on $\psi$, we now bound each of these risks separately. 
For each $l=0,1$, following identical calculations as in the corresponding step for the proof of Theorem~\ref{thm:limits_evaluate_stability},
we have
\begin{align*}
&\bbE\left[\ell(\hf_{l,n}'(X_{n+1}'),Y_{n+1}')\cdot \indi_{(X_{n+1}',Y_{n+1}')\neq (x_*,y_*)}\right] \\
&\hspace{.5in}=
(1-c)^{n+1} \bbE\left[\ell(\hf_{l,n}(X_{n+1}'),Y_{n+1}')\mid  (x_*,y_*)\not\in \cD_n',  (X_{n+1}',Y_{n+1}')\neq (x_*,y_*)\right]\\
&\hspace{1in} + (1-c)\left[1-(1-c)^n\right] \cdot \bbE\left[\ell(f_{1-l}(X_{n+1}'),Y_{n+1}')\mid   (X_{n+1}',Y_{n+1}')\neq (x_*,y_*)\right]\\
&\hspace{.5in}=(1-c)^{n+1} R_{P_*,n}(\cA_l) +   (1-c)\left[1-(1-c)^n\right] \cdot R_{P_*}(f_{1-l}).\end{align*}
Since the loss takes values in $[0,B]$, therefore,
\[R_{P,n}(\cA_0') \leq 
(1-c)^{n+1} R_{P_*,n}(\cA_0) +   (1-c)\left[1-(1-c)^n\right]  R_{P_*}(f_1)
+ cB,\]
and 
\[R_{P,n}(\cA_1') \geq 
(1-c)^{n+1} R_{P_*,n}(\cA_1) +   (1-c)\left[1-(1-c)^n\right]  R_{P_*}(f_0).\]
Combining these calculations, then,
\begin{multline*}\Delta_{P',n}(\cA_0',\cA_1')
\\\leq (1-c)^{n+1} \left(R_{P_*,n}(\cA_0) - R_{P_*,n}(\cA_1) \right) +  (1-c)\left[1-(1-c)^n\right]  \left(R_{P_*}(f_1) - R_{P_*}(f_0)\right)
+ cB\\
=  (1-c)^{n+1} \Delta_{P_*,n}(\cA_0,\cA_1) +    (1-c)\left[1-(1-c)^n\right]\Delta_{P_*}(f_1,f_0)
+ cB.\end{multline*}
Following similar calculations as in the proof of Theorem~\ref{thm:limits_evaluate_stability},
we can also calculate
\[\Delta_{P_*,n}(\cA_0,\cA_1) \geq \Delta_{P,n}(\cA_0,\cA_1) - 2B(n+1)\epsilon,\]
and (by antisymmetry of $\psi$)
\[\Delta_{P_*}(f_1,f_0) \leq \Delta_P(f_1,f_0) + 2B\epsilon = - \Delta_P(f_0,f_1) + 2B\epsilon \leq  - \Delta_P^{\max} +\epsilon(1+2B).\]
Therefore,
\[\Delta_{P',n}(\cA_0',\cA_1') \leq (1-c)^{n+1} \Delta_{P,n}(\cA_0,\cA_1) +  (1-c)\left[1-(1-c)^n\right] (-\Delta_P^{\max})
+ cB + (1+2B(n+1))\epsilon.\]
Recalling~\eqref{eqn:assume_c_compare_2}, we therefore see that, for sufficiently small $\epsilon>0$, we have $\Delta_{P',n}(\cA_0',\cA_1')\leq 0$.

\subsubsection{Step 4: verifying the stability}
To complete the proof, we need to verify that $\beta_q(\cA_l',P',n) \leq \gamma_q$ for each $l=0,1$.
In fact, this calculation is identical to the corresponding step in the proof of Theorem~\ref{thm:limits_evaluate_stability}, so we omit the details.

\subsection{Proof of Theorem~\ref{thm:limits_compare_stability_alt}} \label{sec:compare-proof-stability}
Following the definitions established in Section \ref{sec:algorithm-comparison-as-evaluation}, we can see that
\[
	\tilde{\ell}((\hy_0, \hy_1), y) \in [0, 2B], \quad  \tilde{R}_{P,n}(\tilde{\cA}) = B - \Delta_{P,n}(\cA_0, \cA_1)\]
	and
	\[\tilde{R}_P^{\max} := \sup_{(f_0,f_1)} \tilde{R}_P((f_0,f_1)) = \sup_{f_0,f_1}( B - \Delta_P(f_0,f_1)) = \sup_{f_0,f_1}( B + \Delta_P(f_1,f_0)) = B + \Delta_P^{\max}.\]
Now we apply Theorem~\ref{thm:limits_evaluate_stability}, with $2B$ in place of $B$, and with $B$ in place of $\tau$ (since $\tilde{R}_P(\tilde{\cA}) \geq B$ if and only if $\Delta_P(\cA_0,\cA_1) \leq 0$.) We then have
\[\tilde\tau = B\left(1+\frac{\alpha^{-1}-1}{N}\right).\] Plugging in these substitutions, the result of Theorem~\ref{thm:limits_evaluate_stability} tells us that power for evaluating the risk $\tilde{R}_{P,n}(\tilde\cA)$  is bounded as 
\[
	\bbP_P( \hT(\tilde\cA, \cD_N)  = 1) \leq \left[\alpha\left(1 + \frac{ B(1+\frac{\alpha^{-1}-1}{N}) - \tilde{R}_{P,n}(\tilde\cA) }{\tilde{R}_{P}^{\max} - B(1+\frac{\alpha^{-1}-1}{N}) } \right)^{N/n} \right] \wedge 1,
\]
Plugging in all the substitutions calculated above, we then have
\[
	\bbP_P( \hT(\cA_0,\cA_1, \cD_N)  = 1) \leq \left[\alpha\left(1 + \frac{ B(1+\frac{\alpha^{-1}-1}{N}) - (B - \Delta_{P,n}(\cA_0, \cA_1)) }{( B + \Delta_P^{\max}) - B(1+\frac{\alpha^{-1}-1}{N}) } \right)^{N/n} \right] \wedge 1,
\]
which simplifies to the result claimed in Theorem~\ref{thm:limits_compare_stability_alt} and thus completes the proof.

However, in order to formally verify that Theorem~\ref{thm:limits_evaluate_stability} can be applied to this setting, we need to verify that for the comparison problem, the notion of stability,
and the notion of a black-box test, coincide with the definitions for the algorithm evaluation setting. First, we consider stability.
For any paired algorithm $\tilde\cA = (\cA_0,\cA_1)$, let $\tilde\beta_q(\tilde\cA,P,n)$ denote the $\ell_q$-stability of the algorithm $\tilde\cA$ with respect to the loss $\tilde\ell$ (i.e., as in Definition~\ref{def:ell_q_stability}, but 
with $\tilde\ell$ in place of $\ell$). We can then observe that, by definition,
\[\tilde\beta_q(\tilde\cA,P,n) = \beta_q(\cA_0,\cA_1,P,n),\]
where $\beta_q(\cA_0,\cA_1,P,n)$ is the $\ell_q$-stability of the pair $(\cA_0,\cA_1)$ with respect to the comparison function $\psi$ (as in~\eqref{eqn:define_stability_psi}).
Finally, for defining a black-box test, we observe that Definitions~\ref{def:black-box-test} and~\ref{def:black_box_test_comp} are identical except we replace one algorithm in Definition~\ref{def:black-box-test} with two algorithms in Definition~\ref{def:black_box_test_comp}. This finishes the proof. 

\section{Additional proofs}

\subsection{Proof of Proposition~\ref{thm:limits_evaluate-new-validity}}

	Let us first introduce another testing problem so that we can connect the testing problem \eqref{eq:tail-risk-test} with the original testing problem in \eqref{eqn:hypothesis_test}:
	\begin{equation*}
		\tilde{H}_0: R_{P,n}(\cA) \geq \delta B + \epsilon \quad \textnormal{versus} \quad \tilde{H}_1: R_{P,n}(\cA) < \delta B + \epsilon. 
	\end{equation*}
	Then we observe that for any $\cA, P$ such that $R_{P,n}(\cA) \geq \delta B + \epsilon$, we also have $\bbP( R_P(\cA(\cD_n; \xi)) \geq \epsilon ) \geq \delta $. This is because if $\bbP( R_P(\cA(\cD_n; \xi)) \geq \epsilon ) < \delta$, then 
	\begin{equation} \label{ineq:risk-transfer}
		\begin{split}
			R_{P,n}(\cA) &= \bbE[ R_P(\cA(\cD_n; \xi)) \indi\{ R_P(\cA(\cD_n; \xi)) \geq \epsilon \} ] + \bbE[ R_P(\cA(\cD_n; \xi)) \indi\{ R_P(\cA(\cD_n; \xi)) < \epsilon \} ]  \\
			& \leq B \bbP(R_P(\cA(\cD_n; \xi)) \geq \epsilon) +  \epsilon \bbP(R_P(\cA(\cD_n; \xi)) < \epsilon) \\
			& < B \delta + \epsilon,
		\end{split}
	\end{equation} which contradicts with $R_{P,n}(\cA) \geq \delta B + \epsilon$. This implies that for any $\hT$ satisfying \eqref{ineq:new-validity}, it also satisfies 
	\begin{equation} \label{ineq:intermediate-validity}
	\bbP_P( \hT(\cA,\cD_N) = 1)\leq \alpha \textnormal{ for any $\cA,P$ such that $R_{P,n}(\cA) \geq \delta B + \epsilon$}.
\end{equation}
	Thus, by Theorem \ref{thm:limits_evaluate}, for any $\cA$ and any $P$ with $R_{P,n}(\cA)< \delta B + \epsilon$, the $\hT$ should satisfy
	\begin{equation} \label{ineq:power-upper-new}
		\bbP_P( \hT(\cA, \cD_N)  = 1) \leq \left[\alpha\left(1 + \frac{ \tilde\tau - R_{P,n}(\cA) }{R_{P}^{\max} - \tilde\tau } \right)^{N/n} \right] \wedge 1.
	\end{equation}
Finally, for any $\cA$ and any $P$ with $\bbP( R_P(\cA(\cD_n; \xi)) \geq \epsilon ) < \delta$, by \eqref{ineq:risk-transfer}, we have for the same $\cA$ and $P$, it satisfies that $R_{P,n}(\cA) < B \delta + \epsilon$, then the power upper bound in \eqref{ineq:power-upper-new} also holds. This finishes the proof of this proposition.

\subsection{Proof of Theorem~\ref{thm:binomial_power}}
First, we have $S \sim \text{Binomial}( \lfloor N/(n+1)\rfloor, R_{P,n}(\cA))$ by construction. First, we verify the validity of this test. Fix any $\cA$, $P$ with $R_{P,n}(\cA) \geq \tau$. In this case, 
	\begin{equation*}
	\begin{split}
		&\bbP(\hT(\cA, \cD_N) = 1) \\
		&= \bbP(S  < k_* )  + a_* \bbP( S = k_* )\\
		&= (1-a_*)\bbP(S  < k_* )  + a_* \bbP( S \leq  k_* )\\
		&= (1 - a_*) \bbP( \textnormal{Binomial}(\textstyle{\lfloor \frac{N}{n+1}\rfloor}, R_{P,n}(\cA)  )  < k_* ) + a_* \bbP( \textnormal{Binomial}(\textstyle{\lfloor \frac{N}{n+1}\rfloor}, R_{P,n}(\cA)  )  \leq k_* ) \\
		& \leq  (1 - a_*) \bbP( \textnormal{Binomial}(\textstyle{\lfloor \frac{N}{n+1}\rfloor}, \tau  )  < k_* ) + a_* \bbP( \textnormal{Binomial}(\textstyle{\lfloor \frac{N}{n+1}\rfloor}, \tau  )  \leq k_* ) \\
		& =   \bbP( \textnormal{Binomial}(\textstyle{\lfloor \frac{N}{n+1}\rfloor}, \tau  )  < k_* ) + a_* \bbP( \textnormal{Binomial}(\textstyle{\lfloor \frac{N}{n+1}\rfloor}, \tau  )  = k_* )\\
		& = \alpha
	\end{split}
	\end{equation*} where the last step holds by our choice of $k_*,a_*$, while the inequality holds because $R_{P,n}(\cA) \geq \tau$ and so $\textnormal{Binomial}(\textstyle{\lfloor \frac{N}{n+1}\rfloor}, R_{P,n}(\cA)  )$ stochastically dominates $\textnormal{Binomial}(\textstyle{\lfloor \frac{N}{n+1}\rfloor}, \tau  )$.
	
	Next, we compute the power when $\alpha < (1-\tau)^{\lfloor N/(n+1)\rfloor }$. By this bound on $\alpha$ we have $$\bbP( \textnormal{Binomial}(\textstyle{\lfloor \frac{N}{n+1}\rfloor}, \tau  )  = 0 ) = (1-\tau)^{ \lfloor N/(n+1)\rfloor } > \alpha,$$ which implies that
	\[k_* = 0, \ a_* = \frac{\alpha}{\bbP( \textnormal{Binomial}(\textstyle{\lfloor \frac{N}{n+1}\rfloor}, \tau  )  = 0 )}.\]
Therefore, by the construction of the Binomial test $\hT$, we have
\begin{multline*}
\bbP( \hT(\cA, \cD_N) = 1 ) = a_*\bbP(S=0) =  \frac{\alpha}{\bbP( \textnormal{Binomial}(\textstyle{\lfloor \frac{N}{n+1}\rfloor}, \tau  )  = 0 )}\cdot \bbP\left(\textnormal{Binomial}(\textstyle{\lfloor \frac{N}{n+1}\rfloor}, R_{P,n}(\cA)  ) =0\right)\\
=  \frac{\alpha}{(1 - \tau)^{\lfloor N/(n+1)\rfloor}} (1 - R_{P,n}(\cA)) ^{\lfloor N/(n+1)\rfloor}
 = \alpha \left(1 + \frac{\tau - R_{P,n}(\cA) }{1-\tau} \right)^{ \lfloor N/(n+1)\rfloor } ,
\end{multline*}
as claimed.

\subsection{Proof of Proposition~\ref{prop:consistency_regime}}
Let $(X_1,Y_1),\dots,(X_n,Y_n),(X'_1,Y'_1),\dots,(X'_n,Y'_n)\iidsim P$.
Let
\[\hf_n = \cA(\cD_n), \quad \hf_n'=\cA(\cD_n'),\]
where $\cD_n =\{(X_i,Y_i)\}_{i \in [n]}$, $\cD'_n=\{(X'_i,Y'_i)\}_{i \in [n]}$.
(Recall we have assumed $\cA$ is a deterministic algorithm, so we suppress the dependence on the random seed $\xi$.)

Next, we define some additional notation. For each $j=1,\dots,n-1$, define data set
\[\cD_n^{(j)} = \big((X'_1,Y'_1),\dots,(X'_j,Y'_j),(X_{j+1},Y_{j+1}),\dots,(X_n,Y_n)\big),\]
and let $\cD_n^{(0)}=\cD_n$, $\cD_n^{(n)} = \cD_n'$, and define 
$\tilde\cD_n^{(j)}$ be the same with $j$th point removed for each $j=1,\dots,n$, i.e.,
\[\tilde\cD_n^{(j)} = \big((X'_1,Y'_1),\dots,(X'_{j-1},Y'_{j-1}),(X_{j+1},Y_{j+1}),\dots,(X_n,Y_n)\big).\]
Note that for each $j$, $\cD_n^{(j-1)}$ and $\tilde\cD_n^{(j)}$ differ by one data point (i.e., removing $(X_j,Y_j)$),
and $\cD_n^{(j)}$ and $\tilde\cD_n^{(j)}$ differ by one data point (i.e., removing $(X'_j,Y'_j)$). 

For the case $q=1$, we have
\begin{align*}
&\bbE\left[ \left| R_P\big(\cA(\cD_n^{(j-1)})\big) -R_P\big(\cA(\cD_n^{(j)})\big) \right|\right]\\
&\leq \bbE\left[ \left| R_P\big(\cA(\cD_n^{(j-1)})\big) - R_P\big(\cA(\tilde\cD_n^{(j)})\big) \right|\right] + \bbE\left[ \left| R_P\big(\cA(\cD_n^{(j)})\big) - R_P\big(\cA(\tilde\cD_n^{(j)})\big) \right|\right]\\
&\leq \beta_1(\cA,P,n) + \beta_1(\cA,P,n) = 2\beta_1(\cA,P,n),\end{align*}
where the last inequality holds by definition of $\ell_1$-stability. Since this is true for each $j$, then,
\begin{multline*}\bbE\left[\left| R_P(\hf_n) - R_P(\hf_n')\right|\right]
= \bbE\left[ \left| R_P\big(\cA(\cD_n^{(0)})\big) -R_P\big(\cA(\cD_n^{(n)})\big) \right|\right]\\
\leq \sum_{j=1}^n \bbE\left[ \left| R_P\big(\cA(\cD_n^{(j-1)})\big) -R_P\big(\cA(\cD_n^{(j)})\big) \right|\right]
\leq 2n\beta_1(\cA,P,n).\end{multline*}
By Jensen's inequality, we have
\[\bbE\left[ \left|R_P(\hf_n) - R_{P,n}(\cA)\right|\right] = 
\bbE\left[ \left|R_P(\hf_n) - \bbE\left[R_P(\hf_n')\right]\right|\right] \leq 
\bbE\left[ \left|R_P(\hf_n) - R_P(\hf_n') \right|\right] .\]
This completes the proof for $q=1$.

Next, we turn to the case $q=2$. Note that $\cD_n^{(j-1)},\cD_n^{(j)}$ are in fact i.i.d.\ conditional on $\tilde\cD_n^{(j)}$ (i.e.,
each is obtained by drawing a new sample from $P$ to fill the $j$th position in the data set).
Therefore,
\begin{align*}
&\bbE\left[ \left( R_P\big(\cA(\cD_n^{(j-1)})\big) - R_P\big(\cA(\cD_n^{(j)})\big) \right)^2\right]\\
&= \bbE\left[2\textnormal{Var}\left(R_P\big(\cA(\cD_n^{(j-1)})\big) \mid \tilde\cD_n^{(j)}\right)\right]\\
&\leq  \bbE\left[2\bbE\left[\left(R_P\big(\cA(\cD_n^{(j-1)})\big) - R_P\big(\cA(\tilde\cD_n^{(j)})\big)\right)^2 \mid \tilde\cD_n^{(j)}\right]\right]\\
&=2 \bbE\left[\left(R_P\big(\cA(\cD_n^{(j-1)})\big) - R_P\big(\cA(\tilde\cD_n^{(j)})\big)\right)^2\right]\\
&\leq 2\beta_2^2(\cA,P,n),\end{align*}
where the last inequality holds by definition of $\ell_2$-stability. 
This is true for each $j=1,\dots,n$. Therefore, by the Efron--Stein Inequality \citep{efron1981jackknife,steele1986efron},
since $R_{P,n}(\cA)=\bbE[R_P(\hf_n)]$,
it holds that
\[\bbE\left[\left(R_P(\hf_n) - R_{P,n}(\cA)\right)^2\right]\leq \frac{1}{2}\sum_{j=1}^n \bbE\left[ \left( R_P\big(\cA(\cD_n^{(j-1)})\big) -  R_P\big(\cA(\cD_n^{(j)})\big) \right)^2\right]
= n\beta_2^2(\cA,P,n).\]
This completes the proof for $q=2$.

\end{document}